\documentclass[a4paper, 10pt]{article}
\usepackage[pdftex]{graphicx}
\usepackage[pdftex]{color}
\usepackage{amsmath,amsthm,amssymb}
\usepackage[mathscr]{euscript}
\usepackage{yfonts}
\usepackage{makeidx}
\usepackage{stmaryrd}
\usepackage{multicol}
\usepackage{bm}
\usepackage{lscape}
\usepackage{latexsym}
\usepackage[all,frame,line,color]{xy}

\numberwithin{equation}{section}
\newtheorem{thm}{Theorem}[section]

\newtheorem{lem}[thm]{Lemma}
\newtheorem{cor}[thm]{Corollary}
{\bf}{\it}

\newtheorem{fthm}{Theorem}{\bf}{\it}
{\bf}{\it}
{\bf}{\it}
{\bf}{\it}
{\bf}{\it}

\theoremstyle{definition}
\newtheorem{defn}[thm]{Definition}

{\bf}{\rm}

\theoremstyle{remark}

\newtheorem{rem}[thm]{Remark}
{\bf}{\it}

\newtheorem{definition and corollary}[thm]{Definition and Corollary}

{\it}{\rm}

\newcommand{\al}{\alpha}
\newcommand{\af}{\mathrm{af}}

\newcommand{\C}{{\mathbb C}}
\newcommand{\bC}{{\mathbf C}}

\newcommand{\cO}{{\mathcal O}}

\newcommand{\Hom}{\mbox{\rm Hom}}

\newcommand{\codim}{\mathrm{codim}}

\newcommand{\bI}{{\mathbf I}}

\newcommand{\ch}{\mathrm{ch}}
\newcommand{\gch}{\mathrm{gch}}

\newcommand{\la}{\lambda}
\newcommand{\lo}{\mathrm{loc}}

\newcommand{\Gr}{\mathrm{Gr}}

\newcommand{\ra}{\mathrm{rat}}
\newcommand{\pt}{\mathrm{pt}}

\newcommand{\g}{\mathfrak{g}}
\newcommand{\gb}{\mathfrak{b}}
\newcommand{\sB}{\mathscr B}
\newcommand{\sGB}{\mathscr{GB}}

\newcommand{\tI}{\mathtt{I}}
\newcommand{\tJ}{\mathtt{J}}

\newcommand{\bO}{\mathbb{O}}

\renewcommand{\P}{\mathbb{P}}

\newcommand{\bQ}{\mathbf{Q}}

\newcommand{\sQ}{\mathscr{Q}}

\newcommand{\sX}{\mathscr{X}}

\newcommand{\Q}{\mathbb{Q}}
\newcommand{\R}{\mathbb{R}}
\newcommand{\bX}{\mathbb{X}}

\newcommand{\Z}{\mathbb{Z}}

\newcommand{\Gm}{\mathbb G_m}

\title{On quantum $K$-groups of partial flag manifolds}
\author{Syu \textsc{Kato}\footnote{Department of Mathematics, Kyoto University, Oiwake Kita-Shirakawa Sakyo Kyoto 606-8502 JAPAN \tt{E-mail:syuchan@math.kyoto-u.ac.jp}}}

\begin{document}

\maketitle

\begin{abstract}
We show that the equivariant small quantum $K$-group of a partial flag manifold is a quotient of that of the full flag manifold in a way that respects the Schubert classes. This is a $K$-theoretic analogue of the parabolic version of Peterson's theorem [Lam-Shimozono, Acta Math. {\bf 204} (2010)] that exhibits a different behavior from the case of quantum cohomology. Our quotient maps send some of the Novikov variables to $1$, and the geometric meaning of this specialization is unclear in quantum $K$-theory.
\end{abstract}
\renewcommand{\abstractname}{R\'esum\'e}
\begin{abstract}
Nous montrons que la petite \(K\)-th\'eorie quantique \'equivariante d'une vari\'et\'e de drapeaux partielle est un quotient de celle de la vari\'et\'e de drapeaux compl\`ete, de mani\`ere compatible avec les classes de Schubert. Il s'agit d'un analogue en \(K\)-th\'eorie du th\'eor\`eme de Peterson dans le cadre parabolique [Lam-Shimozono, Acta Math. {\bf 204} (2010)], dont le comportement diff\`ere toutefois de celui de la cohomologie quantique. Nos morphismes de quotient sp\'ecialisent certaines variables de Novikov \`a \(1\), mais la signification g\'eom\'etrique de cette sp\'ecialisation demeure obscure en \(K\)-th\'eorie quantique.
\end{abstract}

\section*{Introduction}
Let $G$ be a connected, simply connected, simple algebraic group over $\C$ with a maximal torus $H$ and a Borel subgroup $B$ that contains $H$. For each parabolic subgroup $P \subset G$ that contains $B$, the associated partial flag manifold is denoted by $G / P$. Let $\Gr$ denote the affine Grassmannian of $G$. In this paper, we describe the $H$-equivariant small quantum $K$-group $qK_H ( G / P )$ of $G / P$ as a quotient of the $H$-equivariant small quantum $K$-group $qK_H ( G / B )$ of $G / B$.

Peterson~\cite{Pet97} (on quantum cohomology) claimed that the ring structure of $qH_H(G/P)$ can be recovered from the equivariant homology ring $H_*^H(\Gr)$ with the Pontryagin product, and Lam--Shimozono~\cite{LS10} later established a precise form of this assertion. In this setting, the existence of a ring surjection $qH_H ( G / B ) \twoheadrightarrow qH_H ( G / P )$ follows from a detailed analysis, as demonstrated in~\cite{LL10}.

In~\cite{Kat18c,Kat18d}, we investigated the $K$-theoretic analogue of the aforementioned relation by utilizing the equivariant $K$-theory of a semi-infinite flag manifold (cf.~\cite{KNS17}) as a mediating object, following an idea of Givental~\cite{Giv94}. This perspective offers a clearer understanding of the relationship between the quantum $K$-groups $qK_H ( G / P )$ for varying $P$ via the pushforward maps along morphisms between semi-infinite flag manifolds.

Exploiting this perspective, we prove the following.

\begin{fthm}[$\doteq$ Theorem~\ref{qKquot}]\label{fqKsurj}
There exists a surjective morphism
\[
qK_H ( G / B ) \twoheadrightarrow qK_H ( G / P )
\]
of $K_H(\pt)$-algebras that sends each Schubert class to a Schubert class. Moreover, if $B \subset P' \subset P$ is an intermediate standard parabolic subgroup, then this morphism factors through $qK_H ( G / P' )$.
\end{fthm}

The same proof also works for the noncommutative variant (Corollary~\ref{qqKquot}). In the course of the proof of Theorem~\ref{fqKsurj}, we also prove a few results (Theorem~\ref{Jrat}, Theorem~\ref{qKcomp}, and Corollary~\ref{exact}) that are basic to the geometry of (the parabolic version of) semi-infinite flag manifolds and the associated Richardson varieties.

We emphasize that the existence of this map is inherently of quantum origin and does \emph{not} specialize to give an algebra map $K_H ( G / B ) \to K_H ( G / P )$. Notably, our map specializes certain Novikov variables to $1$. In the quantum cohomology setting, the corresponding constructions amount to specializing certain Novikov variables to $0$; see~\cite{LS10,LL10}. (In particular, this procedure makes sense only in the presence of the finiteness of the quantum $K$-groups~\cite{ACT18,Kat18c}.) Our algebra map has a hybrid nature, combining features of~\cite{LS10} and~\cite{LL10}, whose exact meaning is unclear at the moment. By setting $P=G$ in Theorem~\ref{fqKsurj} (so that $G/P=\pt$), we obtain, for each $P'$, the ring morphism
\[
qK_H(G/P') \to qK_H(\pt) = K_H(\pt)
\]
as presented in Buch--Chung--Li--Mihalcea~\cite[Corollary~10]{BCLM19}.

In light of the $K$-theoretic version of the Peterson isomorphism (conjectured in~\cite{LLMS17} and proved in~\cite[Corollary~C]{Kat18c}), we also obtain a surjective morphism
\begin{equation}
K_H ( \Gr )_\lo  \twoheadrightarrow qK_H ( G / P )_\lo\label{fKP}
\end{equation}
of suitably localized $K_H(\pt)$-algebras (Theorem~\ref{tcomp}). This morphism also sends each Schubert class to a Schubert class, up to a Novikov monomial twist, and hence reinforces the theme developed in~\cite{BM11,LLMS17,BCMP17} and the references therein.

We emphasize that the explicitness of Theorem~\ref{fqKsurj} and~\eqref{fKP} allows us to transfer various multiplication formulas for $qK_H(G/B)$ to the setting of $qK_H(G/P)$; see, for instance,~\cite{LSS10,KNS17}.

The structure of this paper is as follows. In~\S1, we present preliminary results, including equivariant quantum $K$-theory and quasi-map spaces. In~\S2, we cite results from~\cite{Kat18c,Kat18d} to establish that certain Schubert varieties of parabolic quasi-map spaces have rational singularities (Theorem~\ref{Jrat}). We also introduce variants of the equivariant $K$-group $K_H(\bQ_{\tJ}^\ra)$ of the semi-infinite (partial) flag manifold $\bQ_{\tJ}^\ra$, which differ from those in~\cite{KNS17,Kat18c} and are more suited to our purposes (Theorem~\ref{ch-inj} and Theorem~\ref{qKcomp}). In addition, we derive subtle identities arising from tensor products of line bundles (Corollary~\ref{exact}). These reformulations allow us to deduce the equality of structure constants in Theorem~\ref{qKquot} using key observations made in this paper (Lemma~\ref{i-ideal}), together with~\cite{CKS08,BCMP17}. Apart from these new developments, our overall approach follows that of~\cite{Kat18c} with necessary modifications, although we have aimed to present it from a slightly different perspective. We also provide an example for $G=\mathop{SL}(3)$ in~\S\ref{example}.

The work presented here was announced in~\cite{Kat21} and can be seen as a continuation of~\cite{Kat18c,Kat18d}. After the initial version of this paper appeared on arXiv, another proof of Theorem~\ref{fqKsurj} by a completely different method appeared in~\cite{CL22}.

\section{Preliminaries}\label{sec:prelim}
Throughout this paper, a vector space always refers to a $\C$-vector space, and a graded vector space refers to a $\Z$-graded vector space whose graded pieces are finite-dimensional and whose grading is bounded from above. Tensor products are taken over $\C$ unless stated otherwise. We define the graded dimension of a graded vector space as
\[
\mathrm{gdim} \, M := \sum_{i\in \Z} q^i \dim _{\C} M_i \in \Q (\!(q^{-1})\!).
\]
For notational convenience, we set $\C_q^{0} := \C [q^{-1}]$, $\C_{q} := \C [q,q^{-1}]$, and $\bC_{q} := \C (\!(q^{-1})\!)$. As a rule, we suppress $\emptyset$ and associated parentheses in our notation. This convention particularly applies to $\emptyset = \tJ \subset \tI$ frequently used to specify parabolic subgroups.

\subsection{Groups, root systems, and Weyl groups}
We refer the reader to~\cite{CG97, Kum02} for precise expositions of general material presented in this subsection.

Let $G$ be a connected, simply connected simple algebraic group of rank $r$ over $\C$, and let $B$ and $H$ be a Borel subgroup and a maximal torus of $G$ such that $H \subset B$. We set $N$ $(= [B,B])$ to be the unipotent radical of $B$. We denote the Lie algebra of an algebraic group by the corresponding German lowercase letter. The (finite) Weyl group is given by $W := N_G ( H ) / H$. For an algebraic group $E$, we denote its set of $\C [\![z]\!]$-valued points by $E [\![z]\!]$, and its set of $\C (\!(z)\!)$-valued points by $E (\!(z)\!)$ and similarly for other cases. Let $\mathbf I \subset G [\![z]\!]$ be the preimage of $B \subset G$ via the evaluation at $z = 0$ (the Iwahori subgroup of $G [\![z]\!]$). By a slight abuse of notation, we may consider $\bI$ and $G [\![z]\!]$ as group schemes over $\C$ whose $\C$-valued points are given by these.

Let $P := \Hom(H,\Gm)$ be the weight lattice of $H$, and let $\Delta \subset P$ be the set of roots. (Throughout the main body of this paper, we reserve the symbol $P$ without an argument for the weight lattice; our parabolic subgroups will be denoted by symbols of the form $P(\tJ)$.) Denote by $\Delta_+ \subset \Delta$ the set of positive roots, i.e., those that correspond to root subspaces within $\gb$, and by $\Pi \subset \Delta _+$ the set of simple roots. Each $\al \in \Delta_{+}$ defines a reflection $s_\al \in W$. Let $Q^{\vee}$ be the dual lattice of $P$ with a natural pairing $\left< \bullet, \bullet \right> : Q^{\vee} \times P \rightarrow \Z$. Denote by $\Pi^{\vee} \subset Q ^{\vee}$ the set of positive simple coroots, and let $Q_+^{\vee} \subset Q ^{\vee}$ be the non-negative integer span of $\Pi^{\vee}$. For $\beta, \gamma \in Q^{\vee}$, we write $\beta \ge \gamma$ if and only if $\beta - \gamma \in Q^{\vee}_+$. We denote the set of dominant weights by $P_+ := \{ \lambda \in P \mid \left< \alpha^{\vee}, \lambda \right> \ge 0, \hskip 2mm \forall \alpha^{\vee} \in \Pi^{\vee} \}$ and the set of strictly dominant weights by $P_{++} := \{ \lambda \in P \mid \left< \alpha^{\vee}, \lambda \right> > 0, \hskip 2mm \forall \alpha^{\vee} \in \Pi^{\vee} \}$. Let $\tI := \{1,2,\ldots ,r\}$. We fix bijections $\tI \cong \Pi \cong \Pi^{\vee}$ such that $i \in \tI$ corresponds to $\alpha_i \in \Pi$, its coroot $\alpha_i^{\vee} \in \Pi ^{\vee}$, and a simple reflection $s_i = s_{\al_i} \in W$. Let $\{\varpi_i\}_{i \in \tI} \subset P_+$ be the set of fundamental weights, characterized by the relations $\left< \al_i^{\vee}, \varpi_j \right> = \delta_{ij}$.

For a subset $\tJ \subset \tI$, we define $P ( \tJ )$ as the standard parabolic subgroup of $G$ corresponding to $\tJ$. Specifically, we have the inclusions $\gb \subset \mathfrak p (\tJ) \subset \g$, and $\mathfrak p (\tJ)$ contains the root subspace corresponding to $- \alpha_i$ ($i \in \tI$) if and only if $i \in \tJ$. We define the complement of $\tJ$ as $\tJ^c := \tI \setminus \tJ$. Then, the set of characters of $P ( \tJ )$ is identified with $P_{\tJ} := \sum_{i \in \tJ^c} \Z \varpi_i$. We also set $P_{\tJ, +} := \sum_{i \in \tJ^c} \Z_{\ge 0} \varpi_i = ( P_+ \cap P_{\tJ} )$, $\rho_{\tJ} := \sum_{i \in \tJ^c} \varpi_i$, and $P_{\tJ, ++} := ( \rho_{\tJ} + P_{\tJ,+} )$. We set $Q^{\vee} _{\tJ} := \sum_{i \in \tJ^c} \Z \al_i^{\vee}$ and $Q^{\vee} _{\tJ, +} := \sum_{i \in \tJ^c} \Z_{\ge 0} \al_i^{\vee}$. (In the notation of~\cite[\S2.2]{KNS17}, our $Q^{\vee}_{\tJ}$ is denoted by $Q^{\vee}_{\tJ^c}$.) Next, we define $W_\tJ \subset W$ to be the reflection subgroup generated by $\{s_i\}_{i \in \tJ}$. This subgroup is the Weyl group of the semisimple quotient of $P ( \tJ )$.

Let $\Delta_{\af} := \Delta \times \Z \delta \cup \{m \delta \mid m \neq 0 \}$ be the untwisted affine root system of $\Delta$, with its positive part $\Delta_{\af, +}$ such that $\Delta_+ \subset \Delta_{\af, +}$. We set $\alpha_0 := - \vartheta + \delta$, $\Pi_{\af} := \Pi \cup \{ \alpha_0 \}$, and $\tI_{\af} := \tI \cup \{ 0 \}$, where $\vartheta$ is the highest root of $\Delta_+$. We set $W _{\af} := W \ltimes Q^{\vee}$ and call it the affine Weyl group. It is a reflection group generated by $\{s_i \mid i \in \tI_{\af} \}$, where $s_0$ is the reflection with respect to $\alpha_0$. Let $\ell : W_\af \rightarrow \Z_{\ge 0}$ be the length function and let $w_0 \in W$ be the longest element in $W \subset W_\af$. Together with the normalization $t_{- \vartheta^{\vee}} := s_{\vartheta} s_0$ (for the coroot $\vartheta^{\vee}$ of $\vartheta$), we introduce the translation element $t_{\beta} \in W _{\af}$ for each $\beta \in Q^{\vee}$. By a slight abuse of notation, we denote by $W / W_{\tJ}$ the set of minimal length $W_{\tJ}$-coset representatives within $W$.

Let $W_\af^-$ denote the set of minimal length representatives of $W_\af / W$ within $W_\af$. We define
\[
Q^{\vee}_< := \{\beta \in Q^{\vee} \mid \left< \beta, \al_i \right> < 0, \forall i \in \tI \}.
\]

For each $\la \in P_+$, let $L ( \la )$ denote the irreducible $G$-module with a highest $B$-weight $\la$, i.e. $L ( \la )$ has a $B$-eigenvector of $H$-weight $\la$. For a semisimple $H$-module $V$, we set
\[
\ch \, V := \sum_{\la \in P} e^\la \cdot \dim \mathrm{Hom}_H ( \C_\la, V ).
\]
If $V$ is a $\Z$-graded $H$-module in addition, then we define the graded character as
\[
\gch \, V := \sum_{\la \in P, n \in \Z} q^n e^\la \cdot \dim  \mathrm{Hom}_H ( \C_\la, V_n ).
\]

Let $\sB_{\tJ} := G / P ( \tJ )$ denote the (partial) flag manifold of $G$ associated with $\tJ$. The Bruhat decomposition of $\sB_{\tJ}$ is given by
\begin{equation}
\sB_{\tJ} = \bigsqcup _{u \in W / W_{\tJ}} \bO_{\tJ} ( u )\label{Bdec}
\end{equation}
into $B$-orbits such that $\codim_{\sB_{\tJ}} \, \bO_{\tJ} ( u ) = \ell ( u )$ for each $u\in W / W_\tJ \subset W_\af$. We set $\sB_{\tJ} ( u ) := \overline{\bO_{\tJ} ( u )} \subset \sB_{\tJ}$.

For each $\la \in P_{\tJ}$, we have a line bundle $\cO _{\sB_{\tJ}} ( \la )$ such that
\[
H ^0 ( \sB_{\tJ}, \cO_{\sB_{\tJ}} ( \la ) ) \cong L ( - w_0 \la ), \hskip 3mm \cO_{\sB_{\tJ}} ( \la ) \otimes_{\cO_{\sB_{\tJ}}} \cO_{\sB_{\tJ}} ( - \mu ) \cong \cO_{\sB_{\tJ}} ( \la - \mu ) \hskip 5mm \la, \mu \in P_{\tJ,+}.
\]

For each $u \in W / W_{\tJ}$, let $p_u \in \bO_{\tJ} ( u )$ be the unique $H$-fixed point. We normalize $p_u$ (and hence $\bO_{\tJ} ( u )$) such that the restriction of $H ^0 ( \sB_{\tJ}, \cO_{\sB_{\tJ}} ( \la ) )$ to $p_u$ is isomorphic to $\C_{- w_0 u \la}$ for every $\la \in P_{\tJ,+}$. Note that this convention differs from~\cite{Kat18c} by replacing $\la$ with $-w_0\la$. This change of convention also applies to $Q^{\vee}$ in \S\ref{sec:QM} to ensure consistency in the degree counts in Theorem~\ref{Dr}.

\subsection{Quasi-map spaces}\label{sec:QM}

Here we recall basics of quasi-map spaces from~\cite{FM99,FFKM,Kat18d}.

We have isomorphisms $H^2 ( \sB_{\tJ}, \Z ) \cong P_{\tJ}$ and $H_2 ( \sB_{\tJ}, \Z ) \cong Q ^{\vee}_{\tJ}$. This identifies the (integral points of the) nef cone of $\sB_{\tJ}$ with $P_{\tJ,+} \subset P_{\tJ}$ and the effective cone of $\sB_{\tJ}$ with $Q_{\tJ,+}^{\vee}$. A quasi-map $( f, D )$ consists of an algebraic map $f : \P ^1 \rightarrow \sB_{\tJ}$ together with a colored effective divisor
\[
D = \sum_{x \in \P^1 (\C)} \beta_x \otimes [x] \in Q^{\vee}_{\tJ} \otimes_\Z \mathrm{Div} \, \P^1 \hskip 5mm \beta_x \in Q^{\vee}_{\tJ,+}.
\]
We refer to $D$ as the defect of $(f, D)$, and we define the total defect of $(f,D)$ by
\[
|D| := \sum_{x \in \P^1 (\C)} \beta_x \in Q_{\tJ,+}^{\vee}.
\]

For each $\beta \in Q_{\tJ,+}^{\vee}$, we set
$$\sQ ( \sB_{\tJ}, \beta ) := \{ (f,D) \mid (f,D)\ \text{is a quasi-map and } f_* [ \P^1 ] + | D | = \beta \},$$
where $f_* [\P^1]$ is the class of the image of $\P^1$ multiplied by the degree of $\P^1 \to \mathrm{Im} \, f$. We denote $\sQ ( \sB_{\tJ}, \beta )$ by $\sQ_{\tJ} ( \beta )$ in case there is no danger of confusion. By construction, the space $\sQ_{\tJ} ( \beta )$ admits a $G$-action induced from $\sB_{\tJ}$.

\begin{defn}[Drinfeld-Pl\"ucker data]\label{Zas}
Consider a collection $\mathcal L = \{( \psi_{\la}, \mathcal L^{\la} ) \}_{\la \in P_{\tJ,+}}$ of inclusions $\psi_{\la} : \mathcal L ^{\la} \hookrightarrow L ( \la ) \otimes \mathcal O _{\P^1}$ of line bundles $\mathcal L ^{\la}$ over $\P^1$. The data $\mathcal L$ is called Drinfeld-Pl\"ucker data (DP-data) if the canonical inclusion of $G$-modules
$$\eta_{\la, \mu} : L ( \la + \mu ) \hookrightarrow L ( \la ) \otimes L ( \mu )$$
induces an isomorphism
$$\eta_{\la, \mu} \otimes \mathrm{id} : \psi_{\la + \mu} ( \mathcal L ^{\la + \mu} ) \stackrel{\cong}{\longrightarrow} \psi _{\la} ( \mathcal L^{\la} ) \otimes_{\cO_{\P^1}} \psi_{\mu} ( \mathcal L^{\mu} )$$
for every $\la, \mu \in P_{\tJ,+}$.
\end{defn}

\begin{thm}[Drinfeld, see~\cite{FM99,BFGM} and~\cite{Kat18d}]\label{Dr}
The space $\sQ _{\tJ} ( \beta )$ is isomorphic to the variety formed by isomorphism classes of the DP-data $\mathcal L = \{( \psi_{\la}, \mathcal L^{\la} ) \}_{\la \in P_{\tJ,+}}$ such that $\deg \, \mathcal L ^{\la} = - \left< \beta, \la \right>$.
\end{thm}

For each $u \in W / W_\tJ$, let $\sQ_{\tJ} ( \beta, u ) \subset \sQ_{\tJ} ( \beta )$ be the closure of the set formed by quasi-maps that are defined at $z = 0$ and whose value at $z=0$ is contained in $\sB_{\tJ}(u)\subset\sB_{\tJ}$. (Hence, we have $\sQ_{\tJ} ( \beta ) = \sQ_{\tJ} ( \beta, e )$.) When convenient, we fix a (possibly non-minimal) representative $u\in W$ of the coset $u\in W/W_\tJ$ and still denote it by $u$.

For each $\la \in P_{\tJ}$, $\beta \in Q_{\tJ,+}^{\vee}$, and $u \in W$, we have a $G$-equivariant line bundle $\cO _{\sQ_{\tJ} ( \beta, u )} ( \la )$ obtained by the (tensor product of the) pull-backs $\cO _{\sQ_{\tJ} ( \beta, u )}( \varpi_i )$ of the $i$-th $\cO ( 1 )$ via the embedding
\begin{equation}
\sQ_{\tJ} ( \beta, u ) \hookrightarrow \prod_{i \in \tJ^c} \P ( L ( \varpi_i ) \otimes_{\C} \C [z] _{\le \left< \beta, \varpi_i \right>} )\label{Pemb}
\end{equation}
for each $\beta \in Q_{\tJ,+}^{\vee}$. Using this, we set
$$\chi_q ( \sQ_{\tJ} ( \beta, u ), \cO_{\sQ_{\tJ} ( \beta, u )} ( \la ) ) := \sum_{i \ge 0} (-1)^i \gch \, H^i ( \sQ_{\tJ} ( \beta, u ), \cO_{\sQ_{\tJ} ( \beta, u )} ( \la ) ) \in \C_{q}^0 P$$
for each $\beta \in Q^{\vee}_{\tJ}$ and $\la \in P_{\tJ}$, where the grading $q$ is understood to count the degree of $z$ detected by the $\Gm$-action, and $\C_{q}^0 P$ denotes the group ring of $P$ over $\C_q^0$. Here we understand that
\begin{equation}
\chi_q ( \sQ_{\tJ} ( \beta, u ), \cO_{\sQ_{\tJ} ( \beta, u )} ( \la ) ) = 0 \hskip 5mm \text{if } \beta \not\in Q^{\vee}_{\tJ,+}.\label{Jnondom}
\end{equation}

\begin{rem}
One may twist the $\Gm$-linearization on $\cO_{\sQ_{\tJ}(\beta,u)}(\la)$, which shifts the $q$-grading.
We always take the natural linearization coming from~\eqref{Pemb}; with this choice, the relevant cohomology groups are concentrated in nonnegative $\Gm$-degrees, so the graded characters lie in $\C_q^0 P$ (see~\cite[Corollary~C]{Kat18c}).
\end{rem}

\subsection{Graph and map spaces and their line bundles}\label{GQL}

For each non-negative integer $n$ and $\beta \in Q^{\vee}_{\tJ,+}$, we set $\sGB_{\tJ, n, \beta}$ to be the space of stable maps of genus zero curves with $n$-marked points to $( \P^1 \times \sB_{\tJ} )$ of bidegree $( 1, \beta )$, which is also called the graph space of $\sB_{\tJ}$. A point of $\sGB_{\tJ, n, \beta}$ is a genus zero quasi-stable curve $C$ with $n$-marked points, together with a map to $\P^1$ of degree one. Hence, we have a unique $\P^1$-component of $C$ that maps isomorphically onto $\P^1$. We call this component the main component of $C$ and denote it by $C_0$. The space $\sGB_{\tJ, n, \beta}$ is a normal projective variety by~\cite[Theorem~2]{FP95} and has at worst quotient singularities arising from automorphisms of the source curves. The natural $( \Gm \times H )$-action on $( \P^1 \times \sB_{\tJ} )$ induces a natural $( \Gm \times H )$-action on $\sGB_{\tJ, n, \beta}$.

We have a morphism $\pi_{\tJ, n, \beta} : \sGB_{\tJ, n, \beta} \rightarrow \sQ _{\tJ}( \beta )$ that factors through $\sGB_{\tJ, 0, \beta}$ (Givental's main lemma~\cite{Giv96}; see~\cite[\S8]{FFKM} and~\cite[\S 1.3]{FP95}). Let $\widetilde{\mathtt{ev}}_j : \sGB_{\tJ, n, \beta} \to \P^1 \times \sB_\tJ$ ($1 \le j \le n$) be the evaluation at the $j$-th marked point, and let $\mathtt{ev}_j : \sGB_{\tJ, n, \beta} \to \sB_{\tJ}$ be its composition with the second projection.

The following result is responsible for the basic case (the case of $\tJ = \emptyset$) of our computation:

\begin{thm}[Braverman-Finkelberg~\cite{BF14a,BF14b,BF14c}, see also~\cite{Kat18c} \S4.1]\label{rat-res}
The morphism $\pi_{0,\beta}$ is a rational resolution of singularities.
\end{thm}

We note that $\sGB_{\tJ,n,\beta}$ is irreducible (\cite{KP01}).

For each $\la \in P_{\tJ}$, we have a line bundle
$\cO_{\sGB_{\tJ,n,\beta}}(\la):=\pi_{\tJ,n,\beta}^*\cO_{\sQ_{\tJ}(\beta)}(\la)$.
For a $(\Gm\times H)$-equivariant coherent sheaf $\mathcal F$ on a projective $(\Gm\times H)$-variety $\mathcal X$, let
$\chi_q(\mathcal X,\mathcal F)\in \C_q P$
denote its $(\Gm\times H)$-equivariant Euler--Poincar\'e characteristic, valued in the group ring of $P$ over $\C_q$; this notation enhances the element
$\chi_q(\sQ_{\tJ}(\beta,u),\cO_{\sQ_{\tJ}(\beta,u)}(\la))$
defined in~\S\ref{sec:QM}.

\subsection{Equivariant quantum $K$-group of $\sB_{\tJ}$}\label{eqK}

We introduce a polynomial ring $\C Q^{\vee}_{\tJ,+}$ with its variables $Q_i = Q^{\al_i^{\vee}}$ ($i \in \tJ^c$). We set $Q^{\beta} := \prod_{i \in \tJ^c} Q_i ^{\left< \beta, \varpi_i \right>}$ for each $\beta \in Q^{\vee}_{\tJ}$. We define the $H$-equivariant (small) quantum $K$-group of $\sB_{\tJ}$ as:
\begin{equation}
qK_H ( \sB_{\tJ} ) := K_H ( \sB_{\tJ} ) \otimes \C Q^{\vee}_{\tJ,+},\label{qKdef}
\end{equation}
where $K_H ( \sB_{\tJ} )$ is the complexified $H$-equivariant $K$-group of $\sB_{\tJ}$.

Thanks to (the $H$-equivariant versions of)~\cite{Giv00,Lee04} and the finiteness of the quantum multiplication~\cite{ACT18}, $qK_H ( \sB_{\tJ} )$ is equipped with the commutative and associative product $\star$ (called the quantum multiplication) such that:
\begin{enumerate}
\item the element $[\cO_{\sB_{\tJ}}] \otimes 1 \in qK_H ( \sB_{\tJ} )$ is the identity (with respect to $\cdot$ and $\star$);
\item the map $Q^{\beta}\star$ $(\beta \in Q^{\vee}_{\tJ,+})$ is the multiplication of $Q^{\beta}$ in the RHS of~\eqref{qKdef};
\item we have $\xi \star \eta \equiv \xi \cdot \eta \mod ( Q_i ; i \in \tJ^c )$ for every $\xi,\eta \in K_H ( \sB_{\tJ} ) \otimes 1$.
\end{enumerate}

We set
\[
qK_{\Gm \times H} ( \sB_{\tJ} ) := K_H ( \sB_{\tJ} ) \otimes \C_q Q^{\vee}_{\tJ,+} \hskip 3mm \text{and} \hskip 3mm qK_{\Gm \times H}^{\wedge} ( \sB_{\tJ} ) := K_H ( \sB_{\tJ} ) \otimes \bC_q [\![Q^{\vee}_{\tJ,+}]\!],
\]
where we may regard $\C_q Q^{\vee}_{\tJ,+}$ as a subring of $\bC_q[\![Q^{\vee}_{\tJ,+}]\!]$ via extension of scalars from $\C_q$ to $\bC_q$, followed by $(Q_i \mid i \notin \tJ^c)$-adic completion. We can localize $qK_H ( \sB_{\tJ} )$ (resp. $qK_{\Gm \times H} ( \sB_{\tJ} )$ and $qK_{\Gm \times H}^{\wedge} ( \sB_{\tJ} )$) in terms of $\{Q^{\beta}\}_{\beta \in Q^{\vee}_{\tJ,+}}$ to obtain a ring $qK_H ( \sB_{\tJ} )_\lo$ (resp. vector spaces $qK_{\Gm \times H} ( \sB_{\tJ} )_\lo$ and $qK_{\Gm \times H}^{\wedge} ( \sB_{\tJ} )_\lo$).

We sometimes identify $K_H ( \sB_{\tJ} )$ with the submodule $K_H ( \sB_{\tJ} ) \otimes 1$ of $qK_H ( \sB_{\tJ} )$ or $qK_{\Gm \times H} ( \sB_{\tJ} )$. We set $p_i := [\cO_{\sB_{\tJ}} ( \varpi_i )]$ for $i \in \tJ^c$, and we sometimes consider it as an endomorphism of $qK_{\Gm \times H} ( \sB_{\tJ} )$ through the scalar extension of the product of $K_H ( \sB_{\tJ} )$ (i.e. the classical product). For each $i \in \tJ^c$, let $q^{Q_i \partial_{Q_i}}$ denote the $\bC_q P$-endomorphism of $qK_{\Gm \times H} ( \sB_{\tJ} )$ such that
\[
q^{Q_i \partial_{Q_i}} ( \xi \otimes Q^{\beta} ) = q ^{\left< \beta, \varpi_i \right>} \xi \otimes Q^{\beta} \hskip 5mm \xi \in K_H ( \sB_{\tJ} ), \beta \in Q^{\vee}_{\tJ,+}.
\]

Following~\cite[\S2.4]{IMT15}, we consider the operator $T \in \mathrm{End}_{\bC_q P} \, qK_{\Gm \times H}^{\wedge} ( \sB_{\tJ} )$ (also obtained from the operator $T = T(q,t)$ in~\cite{IMT15} by specializing the parameter $t \in K(\sB_{\tJ})$ to $0$). Then, we have the shift operator (also obtained from the operator $A_i = A_i ( q, t )$ in~\cite{IMT15} by specializing $t$ to $0$) defined by
\begin{equation}
A_i ( q ) = T^{-1} \circ p_i^{-1} q^{Q_i \partial_{Q_i}} \circ T \in \mathrm{End} \, qK^{\wedge}_{\Gm \times H} ( \sB_{\tJ} ) \hskip 5mm i \in \tJ^c.\label{sShift}
\end{equation}

\begin{thm}[\cite{IMT15} and~\cite{ACT18}]\label{smult}
For $i \in \tJ^c$, the operator $A_i ( 1 )$ is well-defined and defines the $\star$-multiplication by $[\cO_{\sB_{\tJ}} ( - \varpi_i )]$ in $qK_H ( \sB _{\tJ} )$.
\end{thm} 

\begin{proof}
The well-definedness of the substitution $q = 1$ is by~\cite[Remark~2.14]{IMT15}. By~\cite[Corollary~2.9]{IMT15} and~\cite[Theorem~8]{ACT18}, the set $\{A_i ( 1 )\}_{i \in \tJ^c}$ defines mutually commutative endomorphisms of $qK_H ( \sB _{\tJ} )$ that commute with the $\star$-multiplication. Since $\mathrm{End} _R R \cong R$ for every ring $R$, we conclude the assertion by $A_i (1)( [\cO_{\sB_{\tJ}}] ) = [\cO_{\sB_{\tJ}} ( - \varpi_i )]$ (\cite[Lemma~6]{ACT18}).
\end{proof}

\section{A description of the quantum $K$-groups}

We continue to work in the setting of the previous section.

\subsection{$K$-groups of semi-infinite partial flag manifolds}

Let $\tJ \subset \tI$ be a subset. The semi-infinite partial flag manifold $\bQ_{\tJ} ^{\ra}$ is an ind-scheme whose set of $\C$-valued points is
\[
G (\!(z)\!) / H ( \C ) \cdot ( [P ( \tJ ), P ( \tJ )] (\!(z)\!) ).
\]
This is a pure ind-scheme of ind-infinite type~\cite{Kat18d}. Note that the group $Q^{\vee} \subset H (\!(z)\!) / H$ acts on $\bQ_{\tJ}^{\ra}$ from the right, whose action factors through $Q^{\vee}_{\tJ}$ via the projection described below. The ind-scheme $\bQ_{\tJ}^{\ra}$ is equipped with a $G [\![z]\!]$-equivariant line bundle $\cO _{\bQ_{\tJ}^{\ra}} ( \la )$ for each $\la \in P_{\tJ}$. Here we normalize so that $\Gamma ( \bQ_{\tJ}^{\ra}, \cO_{\bQ_{\tJ}^{\ra}} ( \la ) )$ is co-generated by its $H$-weight $(-\la)$-part as a $B^- [\![z]\!]$-module.

The following two results are not recorded in the literature in a strict sense, but they are straightforward consequences of the set-theoretic considerations that are allowed in view of~\cite[Theorem~A]{Kat18d}.

\begin{thm}\label{si-Bruhat}
Let $\tJ \subset \tI$. We have an $\bI$-orbit decomposition
$$\bQ_{\tJ}^{\ra} = \bigsqcup_{u \in W / W_\tJ, \beta \in Q^{\vee}_{\tJ}} \bO_{\tJ} ( u t_{\beta} ).$$
\end{thm}

\begin{cor}\label{Iproj}
Let $\tJ \subset \tI$, and let $w = u t_{\beta}$ with $u \in W$ and $\beta \in Q^{\vee}$. We define
\[
[u t_{\beta}]_{\tJ} := u' t_{\beta'} \in W_\af \hskip 10mm [\beta]_{\tJ} := \beta' \in Q^{\vee}_{\tJ},
\]
where $u'$ is the minimal-length representative of the coset $uW_{\tJ}$ and
$$\beta' := \beta - \sum_{j \in \tJ} \left< \beta, \varpi_j \right> \al_j^{\vee}.$$
Then, the natural quotient map $\bQ^{\ra} \rightarrow \bQ_{\tJ}^{\ra}$ sends the $\bI$-orbit $\bO ( u t_{\beta} )$ to $\bO_{\tJ} ( [u t_{\beta}]_{\tJ} )$.
\end{cor}

For each $u \in W / W_\tJ$ and $\beta \in Q^{\vee}_\tJ$, we set $\bQ_\tJ ( u t _{\beta} ) := \overline{\bO_{\tJ} ( u t_{\beta} )} \subset \bQ_{\tJ}^\ra$. We have embeddings $\sB_{\tJ} ( u ) \subset \sQ_{\tJ} ( \beta, u ) \subset \bQ_{\tJ} ( u )$ ($u \in W / W_{\tJ}$) such that the line bundles $\cO ( \la )$ ($\la \in P_{\tJ}$) correspond to each other by restrictions (\cite{BF14b,Kat18,KNS17}).

The following result is essentially contained in~\cite{Kat18d}.

\begin{thm}\label{char-limit}
Let $\tJ \subset \tI$. For each $u \in W / W_\tJ$ and $\la \in P_{\tJ, +}$, we have
\begin{equation}
\lim_{\beta \to \infty} \chi_q ( \sQ_{\tJ} ( \beta,u ), \cO_{\sQ_{\tJ} ( \beta,u )} ( \la ) )
= \gch \, H^0 ( \bQ_{\tJ} ( u ), \cO_{\bQ_{\tJ} ( u )} ( \la ) ) \in \bC_q P.
\label{chi-limit}
\end{equation}
\end{thm}

\begin{proof}
By~\cite[Corollary~C]{Kat18d}, the limit
\[
\lim_{\beta \to \infty} \chi_q ( \sQ_{\tJ} ( \beta,u ), \cO_{\sQ_{\tJ} ( \beta,u )} ( \la ) )
\]
exists. Its value is the graded character of the dual of the Demazure module
${\mathbb W}_{uw_0} ( \la ) \subset \bX ( \la )$ (see~\cite[\S3.2]{Kat18d} for notation).
This follows from the normality of the ring $R( \tJ )$ (\cite[Proposition~A.1]{Kat18d}) and the definition of $R_{u}^{t_{\beta}}(\tJ)$ in~\cite[\S4.1]{Kat18d}. Therefore, we obtain~\eqref{chi-limit} as
\[
\lim_{\beta \to \infty} \chi_q ( \sQ_{\tJ} ( \beta,u ), \cO_{\sQ_{\tJ} ( \beta,u )} ( \la ) )
= \gch \, {\mathbb W}_{uw_0} ( \la )^{\vee}
= \gch \, H^0 ( \bQ_{\tJ} ( u ), \cO_{\bQ_{\tJ} ( u )} ( \la ) )
\]
by~\cite[Corollary~A.3]{Kat18d}; cf.~\cite[Theorem~4.30]{KNS17} and Theorem~\ref{coh-si} below.
\end{proof}

\begin{thm}[\cite{KNS17} Corollary~4.31 and~\cite{Kat18d} Appendix~A]\label{coh-si}
Let $\tJ \subset \tI$. For $w \in W _\af$ and $\la \in P_{\tJ}$, we have
\[
\mathrm{gch} \, H^0 ( \bQ_{\tJ} ( [w]_{\tJ} ), \cO_{\bQ_{\tJ} ( [w]_{\tJ} )} ( \la ) ) = \mathrm{gch} \, H^0 ( \bQ ( w ), \cO_{\bQ ( w )} ( \la ) ) \in \Z_{\ge 0} [\![q^{-1}]\!] P.
\]
This is zero if $\la \not\in P_{\tJ,+}$. Moreover, their higher cohomologies vanish.
\end{thm}

The following is a consequence of the character estimate coming from~\cite[(2.19)]{KNS17} applied to~\cite[Corollary~4.31 and Proposition~D.1]{KNS17}:

\begin{cor}\label{char-est}
For $u \in W$, $\beta \in Q^{\vee}$, and $\la \in P_+$, we have
$$q^{\left< \beta,\la \right>} \gch \, H^0 ( \bQ ( u t_{\beta} ), \cO_{\bQ ( u t_{\beta})} ( \la ) ) \in e^{-u\la} ( 1 + \sum_{0 \neq \gamma \in \Z_{\ge 0} \Pi} \Z e^{\gamma} ) + q^{-1} \C [\![q^{-1}]\!] P.$$
\end{cor}

We define a $\C_q^0 P$-module $K_{\Gm \times H}''(\bQ_{\tJ}^{\ra})$, a naive version of the equivariant $K$-group of $\bQ_{\tJ}^{\ra}$, as follows:
\[
K_{\Gm \times H}'' ( \bQ_{\tJ}^{\ra} ) := \{ \sum_{u \in W / W_{\tJ}, \beta \in Q^{\vee}_{\tJ}} a_{u, \beta} [\cO_{\bQ_{\tJ} ( u t_{\beta} )}] \mid a_{u, \beta} \in \C _{q}^0 P \}.
\]
Here we remark that the sum in the definition of $K_{\Gm \times H}'' ( \bQ_{\tJ}^{\ra} )$ is finite. We set $K_{\Gm \times H}' ( \bQ_{\tJ}^{\ra} ) := \C_q \otimes_{\C_q^0} K_{\Gm \times H}'' ( \bQ_{\tJ}^{\ra} )$. For each $\gamma \in Q^{\vee}_{\tJ}$, we also define
\[
K_{\Gm \times H}'' ( \bQ_{\tJ} ( t_{\gamma} ) ) := \{ \sum_{u \in W / W_{\tJ}, \beta - \gamma \in Q^{\vee}_{\tJ,+}} a_{u, \beta} [\cO_{\bQ_{\tJ} ( u t_{\beta} )}] \in K_{\Gm \times H}'' ( \bQ_{\tJ}^{\ra} ) \}.
\]
We  also consider its completion
\[
K_{\Gm \times H}^{\wedge} ( \bQ_{\tJ}^{\ra} ) := \C_q \otimes_{\C_q^0} \varprojlim _{\gamma} \bigl( K_{\Gm \times H}'' ( \bQ_{\tJ}^{\ra} ) / K_{\Gm \times H}'' ( \bQ_{\tJ} ( t_{\gamma} ) ) \bigr)
\]
and its subset
\[
K_{\Gm \times H}^+ ( \bQ_{\tJ}^{\ra} ) := \{ \!\!\! \sum_{u \in W / W_{\tJ}, \beta \in Q^{\vee}_{\tJ}} \!\!\! a_{u, \beta} [\cO_{\bQ_{\tJ} ( u t_{\beta} )}] \in K_{\Gm \times H}^{\wedge} ( \bQ_{\tJ}^{\ra} ) \mid \sum_{u,\beta} |a_{u,\beta}| q^{-\left< \beta, \rho_\tJ \right>} \in \bC_q P \},
\]
where we understand that
\[
|a| = \sum_{m \in \Z, \la \in P} | a_{m,\la} | \cdot q^m e^{\la} \hskip 5mm \text{when} \hskip 5mm a = \sum_{m \in \Z, \la \in P} a_{m,\la} \cdot q^m e^{\la} \hskip 5mm (a_{m,\la}\in \C).
\]
Note that $K_{\Gm \times H}^+ ( \bQ_{\tJ}^{\ra} )$ and $K_{\Gm \times H} ^{\wedge} ( \bQ_{\tJ}^{\ra} )$ are $\C_q P$-modules.

Let $\mathrm{Fun}_{P_\tJ} ( \bC_q P )$ denote the set of functions on $P_\tJ$ valued in $\bC_q P$. We set
\[
\mathrm{Fun}_{P_\tJ}^{\mathrm{neg}} ( \bC_q P ) := \{f \in \mathrm{Fun}_{P_\tJ} ( \bC_q P ) \mid \exists \gamma \in P_{\tJ} \text{ s.t. } f ( \la ) = 0 \text{ for each } \la \in \gamma + P_{\tJ, +}  \}
\]
and $\mathrm{Fun}_{P_\tJ}^{\mathrm{ess}} ( \bC_q P ) := \mathrm{Fun}_{P_\tJ} ( \bC_q P ) / \mathrm{Fun}_{P_\tJ}^{\mathrm{neg}} ( \bC_q P )$. We define the $\C _{q} P$-linear map
\begin{align}\nonumber
K' _{\Gm \times H} & ( \bQ_{\tJ}^{\ra} ) \ni \sum_{u \in W / W_\tJ, \beta \in Q^{\vee}_{\tJ}} a_{u, \beta} [\cO_{\bQ_{\tJ} ( u t_{\beta} )}] \\
\mapsto & \left( \la \mapsto \sum_{u, \beta} a_{u, \beta} \gch \, H ^0 ( \bQ_{\tJ}^{\ra}, \cO_{\bQ_{\tJ} ( u t_{\beta} )} ( \la ) ) \right) \in \mathrm{Fun}_{P_\tJ} ( \bC_q P ),\label{FJ}
\end{align}
that we denote by $F_{\tJ}$. Here we regard $\cO_{\bQ_{\tJ}(u t_{\beta})}(\la)$ as a sheaf on $\bQ_{\tJ}^{\ra}$ supported on $\bQ_{\tJ}(u t_{\beta})$ via pushforward along the natural closed immersion.

\begin{thm}\label{ch-inj}
Let $\tJ \subset \tI$. The functional $F_{\tJ}$ induces an injective $\C _{q} P$-linear map
\[
K' _{\Gm \times H} ( \bQ_{\tJ}^{\ra} ) \longrightarrow \mathrm{Fun}_{P_\tJ}^{\mathrm{ess}} ( \bC_q P ).
\]
This extends to a $\bC _{q} P$-linear map
\[
K^+ _{\Gm \times H} ( \bQ_{\tJ}^{\ra} ) \longrightarrow \mathrm{Fun}_{P_\tJ}^{\mathrm{ess}} ( \bC_q P ).
\]
\end{thm}

\begin{proof}
The first assertion reduces to the $\C _{q} P$-linear independence of the functionals
\[
P_{\tJ,+} \ni \la \mapsto \mathrm{gch} \, H^0 ( \bQ_{\tJ} ( u t_{\beta} ), \cO_{\bQ_{\tJ} ( u t_{\beta} )} ( \la ) ) \hskip 5mm u \in W / W_{\tJ}, \beta \in Q^{\vee}_{\tJ}.
\]
In view of Theorem~\ref{coh-si}, this follows from Corollary~\ref{char-est} (see also~\cite[Proof of Proposition~5.11]{KNS17}).

We prove the second assertion. Since we pass $F_{\tJ} ( a ) \in \mathrm{Fun}_{P_\tJ} ( \bC_q P )$ for each $a \in K' _{\Gm \times H} ( \bQ_{\tJ}^{\ra} )$ to $\mathrm{Fun}_{P_\tJ}^{\mathrm{ess}} ( \bC_q P )$, we can restrict $F_{\tJ} ( a )$ to $P_{\tJ,++} \subset P_{\tJ}$. The collection of $\C_q P$-coefficients $\{ a_{u, \beta} \}_{u,\beta}$ of an element of $K^+ _{\Gm \times H} ( \bQ_{\tJ}^{\ra} )$ satisfies $a_{u, \beta} = 0$ for $\beta \not\ge \beta_0$ for some $\beta_0 \in Q^{\vee}_{\tJ,+}$, each coefficient $a_{u,\beta}$ is a Laurent polynomial with a uniform upper bound on its $q$-degree, and $\sum_{u,\beta} |a_{u,\beta}| q^{-\left< \beta, \rho_{\tJ} \right>}\in \bC_q P$ by unwinding the definition.

In view of~\cite[Corollary~4.31 and Proposition~D.1]{KNS17} and Theorem~\ref{coh-si}, we have $\left< \beta - \beta_0,\la \right> \ge \left< \beta - \beta_0, \rho_{\tJ} \right>$ ($\Leftrightarrow \left< \beta, - \la + \rho_{\tJ} \right> \le \left< \beta_0, - \la + \rho_{\tJ} \right>$) and
\begin{align}\nonumber
& q^{\left< \beta, \rho_{\tJ} \right>} \mathrm{gch} \, H^0 ( \bQ_\tJ ( u t_{\beta} ), \cO_{\bQ_\tJ ( u t_{\beta} )} ( \la ) ) = q^{\left< \beta, - \la + \rho_{\tJ} \right>} \mathrm{gch} \, H^0 ( \bQ_\tJ ( u ), \cO_{\bQ_\tJ ( u )} ( \la ) )\\
&\le q^{\left< \beta_0, - \la + \rho_{\tJ} \right>}\mathrm{gch} \, H^0 ( \bQ_\tJ ( e ), \cO_{\bQ_\tJ ( e )} ( \la ) ) = q^{\left< \beta_0, \rho_{\tJ} \right>}\mathrm{gch} \, H^0 ( \bQ_\tJ ( t_{\beta_0} ), \cO_{\bQ_\tJ ( t_{\beta_0} )} ( \la ) )\label{ch-est}
\end{align}
for each $\la \in P_{\tJ,++}$, $u \in W$, and $\beta_0\le\beta \in Q^{\vee}_{\tJ,+}$, where the inequality in~\eqref{ch-est} is understood to be coefficient-wise (that are in $\Z_{\ge 0}$) with respect to $\{q^m e^\la\}_{m \in \Z, \la \in P}$. Here the RHS of~\eqref{ch-est} belongs to $\bC_q P$ again by~\cite[Corollary~4.31]{KNS17} (cf.~\cite{BF14b}).
We set
\begin{equation}
a :=\sum_{u,\beta} |a_{u,\beta} | q^{-\left< \beta - \beta_0, \rho_{\tJ} \right>} = q^{\left< \beta_0, \rho_{\tJ} \right>} \sum_{u,\beta} |a_{u,\beta} | q^{-\left< \beta, \rho_{\tJ} \right>}  \in \bC_q P.\label{a-est}
\end{equation}
Using~\eqref{ch-est} and~\eqref{a-est}, we deduce
\[
\sum_{u,\beta} |a_{u,\beta}| \mathrm{gch} \, H^0 ( \bQ_\tJ ( u t_{\beta} ), \cO_{\bQ_\tJ ( u t_{\beta} )} ( \la ) ) \le a \cdot \mathrm{gch} \, H^0 ( \bQ_\tJ ( t_{\beta_0} ), \cO_{\bQ_\tJ ( t_{\beta_0} )} ( \la ) )
\]
for $\la \in P_{\tJ,++}$, which implies the (coefficient-wise) absolute convergence of our functional for each $\la \in P_{\tJ,++}$. This yields the desired map.
\end{proof}

We define the $(\Gm \times H)$-equivariant $K$-group of $\bQ_\tJ^{\ra}$ to be the image of $K^+_{\Gm \times H}(\bQ_\tJ^{\ra})$ in $\mathrm{Fun}_{P_\tJ}^{\mathrm{ess}}(\bC_q P)$, and denote it by $K_{\Gm \times H}(\bQ_\tJ^{\ra})$.

Here the $q=1$ specializations of $K_{\Gm \times H}'(\bQ_{\tJ}^{\ra})$ and $K_{\Gm \times H}(\bQ_{\tJ}^{\ra})$, induced by forgetting the $\Gm$-actions, are well defined since each coefficient of $[\cO_{\bQ_\tJ([w]_\tJ)}]$ ($w \in W_\af$) belongs to $\C_q P$. We denote them by $K_H'(\bQ_{\tJ}^{\ra})$ and $K_H(\bQ_{\tJ}^{\ra})$, respectively.

\begin{thm}[\cite{KNS17} Theorem~6.5 for the case $\tJ = \emptyset$]\label{HQc}
Let $\tJ \subset \tI$. For each $\mu \in P_\tJ$, there exists a $\C_q P$-linear map
\[
\Xi_{\tJ}^+ ( \mu ): K^+ _{\Gm \times H} ( \bQ_{\tJ}^{\ra} ) \longrightarrow K_{\Gm \times H}(\bQ_{\tJ}^{\ra}) \subset \mathrm{Fun}^{\mathrm{ess}}_{P_{\tJ}}(\bC_q P)
\]
defined by the rule
\begin{align*}
& K^+ _{\Gm \times H} ( \bQ_{\tJ}^{\ra} ) \, \ni \sum_{u\in W/W_\tJ,\ \beta\in Q^{\vee}_{\tJ}} a_{u,\beta}\,[\cO_{\bQ_{\tJ}(ut_{\beta})}] \hskip 10mm (a_{u,\beta} \in \C_q P)\\
& \longmapsto \left(\la \mapsto \begin{cases}\sum_{u,\beta} a_{u,\beta}\,\gch\, H^0\bigl(\bQ_{\tJ}^{\ra},\cO_{\bQ_{\tJ}(ut_{\beta})}(\la+\mu)\bigr) & (\la \in P_{\tJ,++}) \\0 & (\la \not\in P_{\tJ,++}) \end{cases} \right) \in \mathrm{Fun}_{P_{\tJ}}(\bC_q P),
\end{align*}
where we view the RHS in $\mathrm{Fun}^{\mathrm{ess}}_{P_{\tJ}}(\bC_q P)$ via the natural surjection
\[
\mathrm{Fun}_{P_{\tJ}}(\bC_q P)\twoheadrightarrow \mathrm{Fun}^{\mathrm{ess}}_{P_{\tJ}}(\bC_q P).
\]
Consequently, $\Xi_{\tJ}^+(\mu)$ gives rise to a $\C_q P$-linear endomorphism $\Xi_{\tJ}(\mu)$ of $K_{\Gm \times H}(\bQ_{\tJ}^{\ra})$.
\end{thm}

\begin{proof}
The reasoning we need is the same as that provided in~\cite[Proof of Theorem~6.5]{KNS17} and~\cite[Proof of Theorem~1.25]{Kat18c} in view of the definition of $K _{\Gm \times H} ( \bQ_{\tJ}^{\ra} )$ and Theorem~\ref{coh-si}.
\end{proof}

\subsection{$\sQ_{\tJ} ( \beta, w )$ has at worst rational singularities}

Let $\sX_{\tJ} (\beta)$ denote the subvariety of $\sGB_{\tJ, 2,\beta}$ such that the first marked point projects to $0 \in \P^1$, and the second marked point projects to $\infty \in \P^1$ under the projection of quasi-stable curves $C$ to the main component $C_0 \cong \P^1$. Let us denote the restrictions of $\mathtt{ev}_i$ $(i=1,2)$ to $\sX_{\tJ} (\beta)$ by the same letter. Since $\sX_{\tJ} (\beta)$ is a normal scheme that has at worst quotient singularities, we may regard it as a smooth stack (\cite{FP95}). As we know that $\sQ_{\tJ} ( \beta )$ is normal (\cite{Kat18d}), we conclude that the restriction of $\pi_{\tJ,2,\beta}$ to $\sX_{\tJ} (\beta)$ also gives a resolution of singularities of $\sQ_{\tJ} ( \beta )$.

For each $\beta \in Q^{\vee}_{\tJ,+}$ and $u \in W / W _{\tJ}$, we set $\sX_{\tJ} (\beta, u) := \mathtt{ev}_1^{-1} ( \sB_{\tJ} ( u ) ) \subset \sX_{\tJ} (\beta)$.

\begin{lem}\label{XJ}
Let $\tJ \subset \tI$. For each $\beta \in Q^{\vee}_{\tJ,+}$ and $u \in W / W _{\tJ}$, the variety $\sX_{\tJ} (\beta, u)$ is projective, normal, and has at worst rational singularities.
\end{lem}

\begin{proof}
Since it is a closed subset of a projective variety $\sGB_{\tJ, 2,\beta}$, it follows that $\sX_{\tJ} (\beta, u)$ is projective. The evaluation map $\mathtt{ev}_1 : \sX_{\tJ} (\beta ) \rightarrow \sB_{\tJ}$ is homogeneous with respect to the $G$-action. Let $N_{\tJ} \subset N$ be the opposite unipotent radical of the conjugation of $P ( \tJ )$ by a lift of $w_0 \in W$ in $N_G ( H )$. By restricting to the open $N$-orbit $N_{\tJ} \times \{ p _e \} \cong \bO_{\tJ} ( e ) \subset \sB_{\tJ}$, we deduce that $\mathtt{ev}_1^{-1} ( \bO_{\tJ} ( e ) ) \cong N_{\tJ} \times \mathtt{ev}_1^{-1} ( p_e )$. By translating using the $G$-action, we conclude that $\mathtt{ev}_1$ is a locally trivial fibration. We know that $\sB_{\tJ} ( u )$ ($u \in W / W_\tJ$) is normal and has at worst rational singularities (see~\cite{KLS14}). Thus, the singularities of $\sX_{\tJ} (\beta, u)$ are locally a product of two rational singularities. From basic properties of rational singularities~\cite[\S5.1]{KM98}, we deduce that having rational singularities is a local condition and it is preserved by taking products. Therefore, we conclude that $\sX_{\tJ} (\beta, u)$ has at worst rational singularities (and the normality is its consequence).
\end{proof}

We have $\sX_{\tJ} (\beta ) = \sX_{\tJ} (\beta, e)$. The map $\pi_{\tJ, 2, \beta}$ restricts to give a $( B \times \Gm)$-equivariant birational proper map
$$\pi_{\tJ, \beta, u} : \sX_{\tJ} (\beta, u) \to \sQ_{\tJ} ( \beta, u )$$
by inspection (cf.~\cite[\S5.2]{Kat18d}). Let $\cO _{\sX_{\tJ} (\beta,u)} ( \la )$ denote the restriction of $\cO _{\sGB_{\tJ, 2, \beta}} ( \la )$ to $\sX_{\tJ} (\beta,u)$ for each $\la \in P_{\tJ}$ and $u \in W / W_{\tJ}$.

\begin{thm}[Koll\'ar~\cite{Kol86} Theorem~7.1]\label{Kol}
Let $f: X \to Z$ be a surjective map between projective varieties, $X$ smooth, and $Z$ normal. Let $F$ be the geometric generic fiber of $f$ and assume that $F$ is connected. The following two statements are equivalent:
\begin{enumerate}
\item $\R^{i} f_* \cO_X = 0$ for all $i > 0$;
\item $Z$ has rational singularities and $H^i( F, \cO_F ) = 0$ for all $i > 0$.
\end{enumerate}
\end{thm}

\begin{thm}\label{Jrat}
Let $\tJ \subset \tI$. For each $\beta \in Q^{\vee}_{\tJ,+}$ and $u \in W / W_{\tJ}$, the variety $\sQ_{\tJ} ( \beta, u )$ has at worst rational singularities. In addition, we have
$$( \pi_{\tJ, \beta, u} )_* \cO_{\sX_{\tJ} (\beta, u)} \cong \cO_{\sQ_{\tJ} ( \beta, u )}, \hskip 5mm \R^{>0} ( \pi_{\tJ, \beta, u} )_* \cO_{\sX_{\tJ} (\beta, u)} \cong \{ 0 \}.$$
\end{thm}

\begin{proof}
By~\cite[Corollary~5.20]{Kat18d}, the variety $\sQ_{\tJ} ( \beta, u )$ is normal. By Lemma~\ref{XJ}, we know that $\sX_{\tJ} ( \beta, u )$ has at worst rational singularities. The same is true for $\tJ = \emptyset$ by~\cite{BF14a,FP95}. The coarse moduli property of $\sX ( \beta )$ yields a morphism $\widetilde{\eta}: \sX ( \beta^+ ) \rightarrow \sX_{\tJ} ( \beta )$ for every $\beta^+ \in Q^{\vee}_+$ such that $\beta = [\beta^+]_{\tJ}$. In view of~\cite[Remark~4.36]{Kat18d} (cf.~Woodward~\cite[Lemma~1]{Woo05}), we can choose $\beta^+$ such that the natural map $\eta : \sQ(\beta^+,u) \rightarrow \sQ_{\tJ}(\beta,u)$ is surjective. (If desired, one may further take $\beta^+$ to be the distinguished choice in~\cite[Lemma~1]{Woo05}; for this choice, $\eta$ enjoys the birationality property described in~\cite[Theorem~3]{Woo05}.)

We have the following commutative diagram:
$$
\xymatrix{
\sX ( \beta^+, u ) \ar[r]^{\widetilde{\eta}} \ar[d]_{\pi_{\beta^+,u}} & \sX _{\tJ} ( \beta, u ) \ar[d]^{\pi_{\tJ,\beta,u}}\\
\sQ ( \beta^+, u ) \ar[r]^{\eta} & \sQ _{\tJ} ( \beta, u ).
}
$$
Here the maps $\pi_{\beta^+,u}$ and $\pi_{\tJ,\beta,u}$ are birational, and $\sX ( \beta^+, u )$ and $\sX _{\tJ} ( \beta, u )$ are projective. Thus, the map $\widetilde{\eta}$ is also surjective by the valuative criterion of properness. Moreover, we have $\R^{\bullet} \eta_{*} \cO_{\sQ ( \beta^+, u )} = \cO_{\sQ_{\tJ} ( \beta, u )}$ by~\cite[Corollary~4.35]{Kat18d}. We find $\R^{\bullet} ( \pi_{\beta^+,u} )_{*} \cO_{\sX ( \beta^+, u )} = \cO_{\sQ ( \beta^+, u )}$ by~\cite[Theorem~4.13]{Kat18c}. By the Leray spectral sequence applied to the composition map $\eta \circ \pi_{\beta^+,u}$, we obtain
\[
\R^{\bullet} ( \eta \circ \pi_{\beta^+,u} )_* \cO_{\sX ( \beta^+, u )} = \cO_{\sQ _{\tJ} ( \beta, u )}.
\]
Hence the geometric generic fiber of $( \eta \circ \pi_{\beta^+,u} )$ is connected (by Stein factorization) and has trivial higher cohomology. Since $\pi_{\tJ,\beta,u}$ is birational, the geometric generic fiber of $( \eta \circ \pi_{\beta^+,u} )$ is the same as that of $\widetilde{\eta}$. Therefore, we conclude
\begin{equation}
\R^{\bullet} \widetilde{\eta}_{*} \cO_{\sX ( \beta^+, u )} = \cO_{\sX_{\tJ} ( \beta, u )}\label{epush}
\end{equation}
by Theorem~\ref{Kol} (by replacing $\sX ( \beta^+, u )$ with its resolution of singularities if necessary, cf.~\cite[Theorem~5.10]{KM98}). By the above commutative diagram, the Leray spectral sequence applied to the composition map $\pi_{\tJ, \beta, u} \circ \widetilde{\eta} = \eta \circ \pi_{\beta^+,u}$ implies
$$\R^{\bullet} ( \pi_{\tJ, \beta, u} )_* \cO_{\sX _{\tJ} ( \beta, u )} \cong \cO_{\sQ_{\tJ} ( \beta, u )}$$
from (\ref{epush}). This shows that $\sQ_{\tJ} ( \beta, u )$ has at worst rational singularities by~\cite[Theorem~5.10]{KM98}.
\end{proof}

\begin{cor}\label{cohXQ}
Keep the setting of Theorem~\ref{Jrat}. For each $\beta \in Q^{\vee}_{\tJ,+}$, $u \in W / W_{\tJ}$, and $\la \in P_{\tJ}$, we have
$$\chi_q ( \sX_{\tJ} (\beta, u), \cO_{\sX_{\tJ} (\beta, u)} ( \la ) ) = \chi_q ( \sQ_{\tJ} (\beta, u), \cO_{\sQ_{\tJ} (\beta, u)} ( \la ) ) \in \C_q P.$$
\end{cor}

\begin{proof}
Apply the projection formula to Theorem~\ref{Jrat}.
\end{proof}

For $\vec{n} = \{ n_i \}_{i \in \tJ^c} \in \Z^{\tJ^c}_{\ge 0}$, we set $x^{\vec{n}} := \prod_{i \in \tJ^c} x_i^{n_i}$. For $\la \in P$, we set $\la[\vec{n}] := \la - \sum_{i \in \tJ^c} n_i \varpi_i$.

\begin{thm}[Iritani-Milanov-Tonita~\cite{IMT15}, cf. Givental-Lee~\cite{GL03}]\label{loc}
Let
\[
\sum_{\beta \in Q^{\vee}_{\tJ, +}, u \in W / W_{\tJ}, \vec{n} \in \Z^{\tJ^c}_{\ge 0}} f_{\beta, u, \vec{n}} ( q ) x^{\vec{n}} Q^{\beta} \in ( \C_q^0 P ) [ \{ x_i \}_{i \in \tJ^c} ][\![Q^{\vee}_{\tJ,+}]\!]
\]
be such that
\begin{equation}
\sum_{\beta \in Q^{\vee}_{\tJ,+}, u \in W / W_{\tJ}, \vec{n} \in \Z^{\tJ^c}_{\ge 0}} f_{\beta, u, \vec{n}} ( q ) \left( \prod_{i \in \tJ^c} A_i(q)^{n_i} \right) Q^{\beta} [\cO_{\sB_\tJ ( u )}]= 0 \in qK_{\Gm \times H}^{\wedge} ( \sB_\tJ ).\label{W-eq-J}
\end{equation}
Here the shift operators $A_i ( q )$ and $A_j (q)$ $(i,j \in \tJ^c)$ commute with each other, and hence~\eqref{W-eq-J} is independent of the chosen order of the factors $A_i(q)$. Then the following identities hold:
\[
\sum_{\beta \in Q^{\vee}_{\tJ,+}, u \in W / W_{\tJ}, \vec{n} \in \Z^{\tJ^c}_{\ge 0}} f_{\beta, u, \vec{n}} ( q ) q^{- \left< \beta, \la[\vec{n}] \right>} \chi_q ( \sX_{\tJ} ( \gamma - \beta, u ), \cO_{\sX_{\tJ} ( \gamma - \beta, u )} ( \la[\vec{n}] ) ) = 0\]
for each $\la \in P_{\tJ,+}$ and $\gamma \in Q^{\vee}_{\tJ,+}$, where we set
\[
\chi_q ( \sX_{\tJ} ( \gamma - \beta, u ), \cO_{\sX_{\tJ} ( \gamma - \beta, u )} ( \la[\vec{n}] ) ) := 0 \hskip 10mm \text{if} \hskip 5mm \gamma - \beta \not\in Q^{\vee}_{\tJ,+}.
\]
\end{thm}

\begin{proof}
The assertion follows by plugging~\eqref{W-eq-J} into~\cite[Proposition~2.20]{IMT15} and observing that $A_i(q)$ corresponds to twisting by $\cO ( - \varpi_i )$ up to $q^{\left< \beta, \varpi_i \right>}$, $Q_i$ twists the Novikov variable (and hence the degree of the stable map spaces), and the effect of $\cO_{\sB_\tJ ( u )}$ is to restrict the whole variety to $\sX_{\tJ} ( \bullet, u )$ via $\mathtt{ev}_1^*$. It can be also seen as a variant of~\cite[Theorem~3.8 and Theorem~3.9]{Kat18c}.
\end{proof}

\subsection{Comparison of equivariant $K$-groups}

\begin{thm}\label{qKcomp}
Let $\tJ \subset \tI$. We have a $\C _q P$-module isomorphism
$$\Psi_{\tJ,q} : qK _{\Gm \times H} ( \sB_{\tJ} )_\lo \stackrel{\cong}{\longrightarrow} K'_{\Gm \times H} ( \bQ_{\tJ}^{\ra} )$$
such that
\begin{enumerate}
\item $\Psi_{\tJ,q} ( [\cO_{\sB_{\tJ} ( u )} ] Q^{\beta} ) = [\cO_{\bQ_{\tJ} ( u t_{\beta} )}]$ for each $u \in W / W_\tJ$ and $\beta \in Q^{\vee}_{\tJ}$;
\item $\Psi_{\tJ,q} ( A_i ( \bullet ) ) = \Xi_{\tJ} ( - \varpi_i ) ( \Psi_{\tJ,q} ( \bullet ) )$ for each $i \in \tJ^c$. In particular, $\Xi_{\tJ} ( - \varpi_i )$ preserves the image of $K'_{\Gm \times H} ( \bQ_{\tJ}^{\ra} )$ in $K_{\Gm \times H} ( \bQ_{\tJ}^{\ra} )$ defined under the natural map
\[
K'_{\Gm \times H} ( \bQ_{\tJ}^{\ra} ) \subset K^+_{\Gm \times H} ( \bQ_{\tJ}^{\ra} ) \longrightarrow K_{\Gm \times H} ( \bQ_{\tJ}^{\ra} ) \subset \mathrm{Fun}_{P_{\tJ}}^{\mathrm{ess}} ( \bC_q P ).
\]
\end{enumerate}
\end{thm}

\begin{proof}
By the definitions of $qK _{\Gm \times H} ( \sB_{\tJ} )_\lo$ and $K'_{\Gm \times H} ( \bQ_{\tJ}^{\ra} )$, we find that $\Psi_{\tJ,q}$ is a $\C _q P$-linear isomorphism.

For each $u \in W / W_{\tJ}$ and $i \in \tJ^c$, we have a finite linear combination
\[
A_i(q) ( [\cO_{\sB_{\tJ} ( u )}] ) = \!\! \sum_{v \in W / W_{\tJ}, \gamma \in Q^{\vee}_{\tJ, +}} \!\! a_{i,u}^{v,\gamma} Q^{\gamma} [\cO_{\sB_{\tJ} ( v )}] \hskip 5mm a_{i,u}^{v,\gamma} \in \C _q P
\]
by~\cite[Appendix~B]{ACT18} and~\cite[Proposition~2.10]{IMT15}. In particular, we have $A_i (q) ( [\cO_{\sB_{\tJ} ( u )}] ) \in qK _{\Gm \times H} ( \sB_{\tJ} )$.

Applying Theorem~\ref{loc} and Corollary~\ref{cohXQ}, we have
\begin{equation}
\chi_q ( \sQ_{\tJ} ( \beta, u ), \cO_{\sQ_{\tJ} ( \beta, u )} ( \la - \varpi_i ) ) = \!\! \sum_{v \in W / W_{\tJ}, \gamma \in Q^{\vee}_{\tJ, +}} \!\! a_{i,u}^{v,\gamma} q^{- \left< \gamma, \la \right>} \chi_q ( \sQ_{\tJ} ( \beta - \gamma, v ), \cO_{\sQ_{\tJ} ( \beta - \gamma, v )} ( \la ) )\label{Aeq}
\end{equation}
for each $\beta \in Q^{\vee}_{\tJ,+}$ and $\la \in P_{\tJ}$. Here the summand on the right-hand side of~\eqref{Aeq} vanishes whenever
$\beta-\gamma \notin Q^{\vee}_{\tJ,+}$ by Theorem~\ref{loc}, in agreement with our convention~\eqref{Jnondom}. Taking the limit $\beta \to \infty$ (cf.~\cite[Proposition~D.1]{KNS17}), we obtain
$$\gch \, H^0 ( \bQ_{\tJ} ( u ), \cO_{\bQ_{\tJ} ( u )} ( \la - \varpi_i ) ) = \! \sum_{v \in W / W_{\tJ}, \gamma \in Q^{\vee}_{\tJ, +}} \! a_{i,u}^{v,\gamma} \gch \, H^0 ( \bQ_{\tJ} ( v t_{\gamma} ), \cO_{\bQ_{\tJ} ( v t_{\gamma} )} ( \la ) )$$
for each $\la \in P_{\tJ,+}$ by Theorem~\ref{chi-limit}. This implies
$$[\cO_{\bQ_{\tJ} ( u )} ( - \varpi_i )] = \! \sum_{v \in W / W_{\tJ}, \gamma \in Q^{\vee}_{\tJ, +}} \! a_{i,u}^{v,\gamma} [ \cO_{\bQ_{\tJ} ( v t_{\gamma} )} ] \in K_{\Gm \times H} ( \bQ_{\tJ}^{\ra} )$$
in view of Theorem~\ref{ch-inj}. Hence, we deduce
\begin{equation}
\Psi_{\tJ,q} ( A_i (  [\cO_{\sB_{\tJ} ( u )}] ) ) = \Xi_\tJ ( - \varpi_i ) ( \Psi_{\tJ,q} ( [\cO_{\sB_{\tJ} ( u )}] ) ) \hskip 5mm u \in W / W_{\tJ},\label{PAu}
\end{equation}
where the equality is in $K_{\Gm \times H}' ( \bQ_{\tJ} ^{\ra} )$. Since $\Psi_{\tJ,q}, A_i$, and $\Xi_\tJ ( - \varpi_i )$ ($i \in \tJ^c$) are $\C _q P$-linear, we conclude the result.
\end{proof}

\begin{cor}[of the proof of Theorem~\ref{qKcomp}]\label{exact}
For each $i \in \tI$, $w \in W_\af$, and $\la \in P$, we have
\[
\gch \, H^0 ( \bQ(w),\cO_{\bQ(w)} ( \la - \varpi_i ) ) = \sum_{v \in W_\af} a_{i,w}^v \gch \, H^0 ( \bQ(v),\cO_{\bQ(v)} ( \la ) ),
\]
where $\{a_{i,w}^v\}_{v \in W_\af} \subset \C_q P$ and the sum on the RHS is finite.
\end{cor}

\begin{proof}
In case $(\la - \varpi_i) \in P_+$, the assertion is obtained as the limit of~\eqref{Aeq} by using Theorem~\ref{char-limit}. Thus, we assume $(\la - \varpi_i) \not\in P_+$ in the below. We have
$$\chi_q ( \sQ_{\tJ} ( \beta ), \cO_{\sQ_{\tJ} ( \beta )} ( \la - \varpi_i ) ) = \chi_q ( \sQ ( \beta^+ ), \cO_{\sQ ( \beta^+ )} ( \la - \varpi_i ) )$$
for large enough $\beta^+ \in Q^{\vee}$ such that $\beta = [\beta^+]_{\tJ} \in Q^{\vee}_{\tJ}$ and the map $\sQ ( \beta^+ ) \to \sQ _{\tJ} ( \beta )$ is surjective. By~\cite[Theorem 3.1]{BF14b} and~\cite{BF14c}, we further deduce
$$\gch\, H^0 ( \sQ_{\tJ} ( \beta ), \cO_{\sQ_{\tJ} ( \beta )} ( \la - \varpi_i ) ) = \gch\, H^0 ( \sQ ( \beta^+ ), \cO_{\sQ ( \beta^+ )} ( \la - \varpi_i ) ) = 0$$
for sufficiently dominant and large enough $\beta^+ \in Q^{\vee}$. By Theorem~\ref{char-limit}, we obtain
\begin{equation}
0 = \gch \, H^0 ( \bQ(e),\cO_{\bQ(e)} ( \la - \varpi_i ) ) = \sum_{v \in W_\af} a_{i,w}^v \gch \, H^0 ( \bQ(v),\cO_{\bQ(v)} ( \la ) ),\label{e-case}
\end{equation}
by taking the limit $\beta \to \infty$ in~\eqref{Aeq} for $\tJ = \emptyset$. This proves the assertion for the case $w=e$. Now we apply the Demazure operators (\cite[Proposition~6.3]{KNS17}) and the translation of $\bI$-orbits via the right $Q^{\vee}$-action on $\bQ^\ra$ to both sides of~\eqref{e-case} to obtain the case of general $w \in W_\af$ (cf.~\cite{Kat18}).
\end{proof}

We have a surjective $\C_q P$-linear morphism
\[
\phi_{\tJ}: K'_{\Gm\times H}(\bQ^{\ra}) \ni [\cO_{\bQ(w)}] \longmapsto [\cO_{\bQ_{\tJ}([w]_{\tJ})}] \in K'_{\Gm\times H}(\bQ_{\tJ}^{\ra}), \qquad w\in W_\af.
\]

\begin{thm}\label{Xi-int}
Let $\tJ \subset \tI$. For each $i \in \tJ^c$, the endomorphism $\Xi ( - \varpi_i )$ descends to an endomorphism $\Xi_{\tJ} ( - \varpi_i )$ on $K' _{\Gm \times H} ( \bQ^{\ra}_{\tJ} )$ through $\phi_{\tJ}$. In addition, the map $\phi_{\tJ}$ induces a surjective $\C P$-module map $K' _H ( \bQ^{\ra} ) \to K' _H ( \bQ_{\tJ}^{\ra} )$ such that $\Xi ( - \varpi_i )$ induces an endomorphism of $K' _{H} ( \bQ_{\tJ}^{\ra} )$. 
\end{thm}

\begin{proof}
Consider the $\C _q P$-linear map defined by
\begin{align*}
K' _{\Gm \times H} & ( \bQ^{\ra} ) \ni \sum_{w \in W _\af} a_{w} [\cO_{\bQ ( w )}] \\
\mapsto & \left( \la \mapsto \sum_{w} a_{w} \gch \, H ^0 ( \bQ^{\ra}, \cO_{\bQ ( w )} ( \la ) ) \right) \in \mathrm{Fun}_{P_\tJ} ( \bC_q P ).
\end{align*}
By Theorem~\ref{coh-si}, this map factors through $K' _{\Gm \times H} ( \bQ^{\ra}_{\tJ} )$ as $[\cO_{\bQ ( w ) }] \mapsto [\cO_{\bQ_{\tJ} ( [w]_{\tJ} ) }]$ for $w \in W_\af$. By Theorem~\ref{qKcomp}, each $\Xi ( - \varpi_i )$ ($i \in \tI$) preserves $K' _{\Gm \times H} ( \bQ^{\ra} )$. In view of Theorem~\ref{ch-inj} and Corollary~\ref{exact}, the endomorphism $\Xi ( - \varpi_i )$ on $K' _{\Gm \times H} ( \bQ^{\ra} )$ descends to an endomorphism of $K' _{\Gm \times H} ( \bQ^{\ra}_{\tJ} )$ for each $i \in \tJ^c$ via the map $\phi_{\tJ}$. By specializing to $q = 1$ (that is possible since the coefficients are in $\C_q$), we conclude that $\phi_{\tJ}$ induces a $\C P$-module surjection $K' _H ( \bQ^{\ra} ) \to K' _H ( \bQ_{\tJ}^{\ra} )$ on which $\Xi ( - \varpi_i )$ descends to an endomorphism.
\end{proof}

By abuse of notation, we denote the surjective map $K' _H ( \bQ^{\ra} ) \to K'_H ( \bQ_{\tJ}^{\ra} )$ in Theorem~\ref{Xi-int} by $\phi_{\tJ}$. We also denote the $q=1$ specializations of the automorphisms $\Xi ( - \varpi_i )$ and $\Xi_{\tJ} ( - \varpi_i )$ in Theorem~\ref{qKcomp} by the same symbols.

\begin{cor}\label{Kcomp}
As the $q = 1$ specialization of Theorem~\ref{qKcomp}, we have a $\C P$-module isomorphism
$$\Psi_{\tJ} : qK _{H} ( \sB_{\tJ} )_\lo \stackrel{\cong}{\longrightarrow} K'_{H} ( \bQ_{\tJ}^{\ra} )$$
such that
\begin{enumerate}
\item $\Psi_{\tJ} ( [\cO_{\sB_{\tJ}( u )} ] Q^{\beta} ) = [\cO_{\bQ_{\tJ} ( u t_{\beta} )}]$ for each $u \in W / W_\tJ$ and $\beta \in Q^{\vee}_{\tJ}$;
\item $\Psi_{\tJ} (  [\cO_{\sB_{\tJ}} ( - \varpi_i )] \star \bullet ) = \Xi_{\tJ} ( - \varpi_i ) ( \Psi_{\tJ} ( \bullet ) )$ for each $i \in \tJ^c$.
\end{enumerate}
\end{cor}

\begin{proof}
Taking Theorem~\ref{qKcomp} into account, it remains to observe that $A_i ( q )$ specializes to $[\cO_{\sB_{\tJ}} ( - \varpi_i )] \star$ by Theorem~\ref{smult}.
\end{proof}

We consider the subring of $qK_H ( \sB_{\tJ} )_{\ge 0} \subset qK_H ( \sB_{\tJ} )$ generated by $\C P$, $\C Q^{\vee}_{\tJ,+}$, and $\{ [\cO_{\sB_{\tJ}} ( - \varpi_i )] \, \star \}_{i \in \tJ^c}$.

\begin{lem}\label{i-ideal}
For each $i \in \tI$, the $\C_q P$-subspace $K_i^q \subset qK_{\Gm \times H} ( \sB )$ spanned by the set
$$\{ [\cO_{\sB ( u )}] Q^{\beta} - [\cO_{\sB ( us_i )}] Q^{\beta'} \mid u \in W, \beta,\beta' \in Q^{\vee}_{+}, \, \text{s.t.} \, \beta - \beta' \in \Z \al^{\vee}_i \}$$
is stable under the action of $A_j ( q )$ $(j \in \tI)$. In particular, its specialization $q = 1$ yields a $\C P$-subspace $K_i \subset qK_{H} ( \sB )$ that is stable under the $qK_H ( \sB )_{\ge 0}$-action.
\end{lem}

\begin{rem}
Lemma~\ref{i-ideal} does not hold if we replace $qK_H ( \sB )$ with $K_H ( \sB )$. We set $G = \mathop{SL} ( 2 )$ (and hence $\sB=\P^1$ and $\tI = \{ 1 \}$). We have an equality $[\cO_{\sB ( s_1 )} ( - \varpi_1 )] = e^{\varpi_1} [\cO_{\sB ( s_1 )}] \in K_H ( \sB )$, which implies
\[
[\cO_{\sB} ( - \varpi_1 )] - [\cO_{\sB ( s_1 )} ( - \varpi_1 )] = e^{- \varpi_1} [\cO_{\sB}] - ( e^{\varpi_1} + e^{-\varpi_1} ) [\cO_{\sB ( s_1 )}] \not\in \C P ( [\cO_{\sB}] - [\cO_{\sB ( s_1 )}] ).
\]
In other words, the vanishing part of Theorem~\ref{coh-si} is crucial in our consideration.
\end{rem}

\begin{proof}[Proof of Lemma~\ref{i-ideal}]
By Theorem~\ref{coh-si}, elements in $\Psi_q ( K_i^q )$ vanish upon applying the functional $F$ defined in~\eqref{FJ} restricted to $\la \in ( P_{\{i\}} + \Z_{\le 0} \varpi_i )$.

Conversely, assume that $a \in \Psi_q ( qK_{\Gm \times H} ( \sB ) )$ vanishes by applying the functional $F$ restricted to $\la \in ( P_{\{i\}} + \Z_{\le 0} \varpi_i )$. Then, $F ( a )$ must also vanish on $\la \in P_{\{i\}}$. We have another finite sum
$$a \equiv \sum_{u \in W / W_{\{i\}}, \beta \in Q^{\vee}_{\{i\}}} a_{u,\beta} [\cO_{\sB (u)}] Q^\beta \mod \Psi^{-1}_q ( K_i^q ) \hskip 5mm a_{u,\beta}\in \Z [q,q^{-1}]$$
obtained from $a$ by summing up the coefficients corresponding to $w \in W_\af$ such that $[w]_{\{i\}} = u t_{\beta}$ to afford $a_{u,\beta}$. Applying Theorem~\ref{ch-inj} for $\tJ = \{i\}$ (or by Corollary~\ref{char-est}), we have that
\[
\sum_{u \in W / W_{\{i\}}, \beta \in Q^{\vee}_{\{i\}}} a_{u,\beta} \cdot \gch \, H^0 ( \bQ (u t_{\beta}), \cO_{\bQ (u t_{\beta})} ( \la ) ) = 0
\]
for every $\la \in P_{\{i\},+}$ implies $a_{u,\beta} = 0$ for every $(u,\beta)$. Thus, we have $a \in \Psi^{-1}_q ( K_i^q )$.

Therefore, $K_i^q$ is precisely the set of elements of $\Psi_q ( qK_{\Gm \times H} ( \sB ) )$ that vanish by applying the functional $F$ restricted to $\la \in ( P_{\{i\}} + \Z_{\le 0} \varpi_i )$. Hence, $K_i^q$ is stable under the action of $\{ \Xi ( - \varpi_j ) \}_{j \in \tI}$ by Corollary~\ref{exact}. It follows that the set $K_i$ is stable by the multiplication by $qK_H ( \sB )_{\ge 0}$.
\end{proof}

\subsection{Comparison between equivariant quantum $K$-groups}

The following crucial observation is due to Buch-Chaput-Mihalcea-Perrin~\cite[\S5]{BCMP17} (see also~\cite[\S1.2]{ACT18}, cf.~\cite[Lemma~4.1.3]{CKS08}):

\begin{itemize}
\item The multiplication rule of $qK_H ( \sB_{\tJ} )$ as a $\C P \otimes \C Q^{\vee}_{\tJ,+}$-algebra is completely determined by the $\star$-multiplication table of $[\cO_{\sB_{\tJ} ( s_i )}]$ for $i \in \tJ^c$.
\end{itemize}

In view of the equality (cf.~\cite[(1.3)]{Kat18c})
$$[\cO_{\sB_\tJ} ( - \varpi_i )] = e^{w_0 \varpi_i} ( [\cO_{\sB_\tJ}] - [\cO_{\sB_\tJ ( s_i )}] ) \in K_H ( \sB_{\tJ} ) \hskip 5mm i \in \tJ^c,$$
we can rephrase this as:
\begin{itemize}
\item The multiplication rule of $qK_H ( \sB_{\tJ} )$ as a $\C P \otimes \C Q^{\vee}_{\tJ,+}$-algebra is completely determined by the $\star$-multiplication table of $\cO_{\sB_{\tJ}} ( - \varpi_i )$ for $i \in \tJ^c$.
\end{itemize}

This follows because, after base change to the fraction field $\C(P\oplus Q^{\vee}_{\tJ})$, the $\C(P\oplus Q^{\vee}_{\tJ})$-algebra $\C(P\oplus Q^{\vee}_{\tJ}) \otimes_{( \C P \otimes \C Q^{\vee}_{\tJ,+} )} qK_H(\sB_{\tJ})$ is generated by
$\{[\cO_{\sB_{\tJ}}(-\varpi_i)]\star\}_{i\in \tJ^c}$; see~\cite[Remark~5.10]{BCMP17}. In other words, we have
\[
\C ( P \oplus Q^{\vee}_\tJ )\otimes_{\C P \otimes \C Q^{\vee}_{\tJ,+}} qK_H ( \sB_\tJ )_{\ge 0} = \C ( P \oplus Q^{\vee}_\tJ ) \otimes_{\C P \otimes \C Q^{\vee}_{\tJ,+}} qK_H ( \sB_\tJ )
\]
and the multiplication rule of $\{ [\cO_{\sB_\tJ} ( - \varpi_i )] \, \star \}_{i \in \tJ^c}$ on some $\C ( P \oplus Q^{\vee}_\tJ )$-basis of $\C ( P \oplus Q^{\vee}_\tJ ) \otimes_{\C P \otimes \C Q^{\vee}_{\tJ,+}} qK_H ( \sB_{\tJ} )$ determines the product structure of $qK_H ( \sB_{\tJ} )$.

\begin{thm}\label{qKquot}
We have a surjective morphism
\[
qK_H ( \sB ) \twoheadrightarrow qK_H ( \sB_{\tJ} )
\]
of commutative algebras such that the image of $[\cO_{\sB ( w )}]$ is $[\cO_{\sB_{\tJ} ( [w]_{\tJ} )}]$ for each $w \in W$, and the image of $Q^{\beta}$ is $Q^{[\beta]_{\tJ}}$ for each $\beta \in Q^{\vee}_+$.
\end{thm}

\begin{proof}
We have the following diagram of $\C P \otimes \C Q^{\vee}_{+}$-modules (with the solid arrows):
\[
\xymatrix{
K_{H}' ( \bQ (e) ) \ar[r]^{\Psi^{-1}} \ar[d]_{\phi_{\tJ}} & qK_H ( \sB ) \ar@{-->}[d]\\
K_{H}' ( \bQ_{\tJ} (e) ) \ar[r]^{\Psi_{\tJ}^{-1}} & qK_H ( \sB_{\tJ} )
}.
\]
The dashed arrow will be defined below. Moreover, the bases correspond via
$\phi_{\tJ} ( [\cO_{\bQ ( w )}] ) = [\cO_{\bQ_\tJ ( [w]_{\tJ} )}]$
for $w \in W \times Q^{\vee}_+ \subset W_\af$. The kernel of the map $\phi_{\tJ}$ is the preimage of the sum of subspaces $K_i$ (borrowed from Lemma~\ref{i-ideal}) for $i \in \tJ$. This defines an ideal of $\Psi ( qK_H ( \sB )_{\ge 0} )$. Therefore, the map $\phi_{\tJ}$ induces an \emph{a priori} $\C P$-algebra structure on
\[
\phi_{\tJ} ( \Psi ( qK_H ( \sB )_{\ge 0})) \subset K_{H}' ( \bQ_{\tJ} ( e ) ).
\]
(If $\tJ=\tI$, then $\phi_{\tJ}([\cO_{\bQ(w)}])\equiv 1$ and $\mathrm{Im}\,\phi_{\tJ}=K_H(\pt)=\C P$. In particular, the induced product on $\mathrm{Im}\,\phi_{\tJ}$ agrees with the usual product on $K_H(\pt)$, and the result follows in this case.) In view of Theorem~\ref{Xi-int}, we conclude
$$\Xi_{\tJ} ( - \varpi_i ) \circ \phi_{\tJ} = \phi_{\tJ} \circ \Xi ( - \varpi_i ) \hskip 5mm i \in \tJ^c.$$
Thus, the above observation and Corollary~\ref{Kcomp} imply that the above module map $\phi_\tJ$ induces an algebra map
$$qK_H ( \sB ) \longrightarrow qK_H ( \sB_{\tJ} )$$
with the desired properties (here we also use that both sides are algebras with respect to the $\star$-product).
\end{proof}

The $\C P$-action commutes with the actions of $A_i ( q )$ ($i \in \tJ^c$), while we have $A_i ( q ) Q^{\beta} = q^{-\left< \beta, \varpi_i \right>} Q^{\beta} A_i ( q )$ for each $i \in \tJ^c$ and $\beta \in Q^{\vee}_{\tJ,+}$ by~\cite[Theorem~A]{Kat18}. In particular, we can localize $qK_{\Gm\times H} ( \sB )_{\ge 0}$ and $qK_{\Gm\times H} ( \sB_{\tJ} )_{\ge 0}$ to the fraction field $\C ( q, P )$ with respect to the (left) $\C_q P$-action, and we can extend the (right) $\C Q^{\vee}_{\tJ,+}$-action to the $\C [\![Q^{\vee}_{\tJ,+}]\!]$-action. Since the proof of Theorem~\ref{qKquot} relies on the comparison of the basis and the actions of the operators $A_i ( q )$, the same reasoning yields the following:

\begin{cor}\label{qqKquot}
We have a surjective $\C_q P$-module morphism
\[
qK_{\Gm\times H} ( \sB ) \twoheadrightarrow qK_{\Gm\times H} ( \sB_{\tJ} )
\]
that intertwines the actions of $A_i ( q )$ $(i \in \tJ^c)$, and the image of $[\cO_{\sB ( w )}]Q^{\beta}$ is $[\cO_{\sB_{\tJ} ( [w]_{\tJ} )}]Q^{[\beta]_{\tJ}}$ for each $w \in W$ and $\beta \in Q^{\vee}_+$. \hfill $\Box$
\end{cor}

\subsection{Comparison with affine Grassmannians}
In this subsection, we deal with an algebra $K_H ( \Gr )$ that can be seen as the $H$-equivariant $K$-group of the affine Grassmannian of $G$ whose product structure is given by the Pontryagin product. For background materials, see~\cite{LSS10,Kat18c}.

For $w \in W_\af^-$, we consider a formal symbol $\Gr_{w}$ and set
$$K_H ( \Gr ) := \bigoplus_{w \in W_\af^-} \C P \, [\cO_{\Gr_{w}}].$$

\begin{thm}[Lam-Schilling-Shimozono, see~\cite{Kat18c} \S1.3]\label{LSS10}
There exists a commutative algebra structure $($whose multiplication is denoted by $\odot)$ on $K_H ( \Gr )$ such that
$$[\cO_{\Gr_{w}}] \odot [\cO_{\Gr_{\beta}}] = [\cO_{\Gr_{w t_\beta}}]$$
for each $w \in W_\af^-$ and $\beta \in Q^{\vee}_<$.
\end{thm}

We call the multiplication $\odot$ of $K_H ( \Gr )$ the {\it Pontryagin product}. Theorem~\ref{LSS10} implies that the set
$$\{ [\cO_{\Gr_{\beta}}] \mid \beta \in Q^{\vee}_< \} \subset ( K_H ( \Gr )_{\lo}, \odot )$$
forms a multiplicative system. We denote by $K_H ( \Gr )_{\lo}$ its localization. The action of an element $[\cO_{\Gr_{\beta}}]$ on $K_H ( \Gr )$ in Theorem~\ref{LSS10} is torsion-free, and hence we have an embedding $K_H ( \Gr ) \hookrightarrow K_H ( \Gr )_\lo$.

\begin{thm}[\cite{Kat18c} Corollary~C]\label{KPfull}
There exists an isomorphism
$$\Phi : ( K _H ( \Gr )_\lo, \odot ) \longrightarrow ( qK_{H} ( \sB )_\lo, \star )$$
of algebras such that
$$\Phi ( [\cO_{\Gr_{ut_{\beta_1}}}] \odot [\cO_{\Gr_{t_{\beta_2}}}]^{-1} ) = [\cO_{\sB ( u )}] Q^{\beta_1 - \beta_2} \hskip 7mm u \in W, \beta_1,\beta_2 \in Q^{\vee}_<.$$
\end{thm}

\begin{thm}\label{tcomp}
There exists a surjective $K_H(\pt)$-algebra homomorphism
$$\eta_\tJ : ( K_H ( \Gr )_\lo, \odot ) \longrightarrow ( qK _H ( \sB_{\tJ} )_\lo, \star )$$
such that
$$\eta_\tJ ( [\cO_{\Gr_{ut_{\beta_1}}}] \odot [\cO_{\Gr_{t_{\beta_2}}}]^{-1} ) = [\cO_{\sB_{\tJ} ( [u]_{\tJ} )}] Q^{[\beta_1 - \beta_2]_{\tJ}} \hskip 7mm u \in W, \beta_1,\beta_2 \in Q^{\vee}_<.$$
\end{thm}

\begin{proof}
Combine Theorem~\ref{KPfull} with Theorem~\ref{qKquot}.
\end{proof}

\section{Examples: $G = \mathop{SL} ( 3 )$}\label{example}
Keep the setting of the previous section with $G = \mathop{SL} ( 3 )$. We have $W = \left< s_1,s_2 \right> \cong \mathfrak S_3$, $P = \Z \varpi_1 \oplus \Z \varpi_2$, and $Q^{\vee} = \Z \al^{\vee} _1 \oplus \Z \al^{\vee} _2$. Recall that $\vartheta := \al_1 + \al_2$ and $\vartheta^{\vee} := \al_1^{\vee} + \al_2^{\vee}$. We have $w_0 = s_1 s_2 s_1 = s_2 s_1 s_2$. In this case, we have three nonempty subsets $\emptyset \neq \tJ \subset \tI = \{ 1,2 \}$. In view of~\cite[Corollary~4.20]{Kat18c}, we may consult~\cite[\S4.2]{LLMS17} (with the convention of $H$-characters twisted by $w_0$) to justify the first equality in each item. The other equalities are consistent with~\cite[\S5.5]{BM11}.

\begin{itemize}
\item We have
$$\hskip -10mm [\cO_{\sB ( s_1 )}] \star [\cO_{\sB ( s_1 )}] = ( 1 - e^{\al_2} ) [\cO_{\sB ( s_1 )}] + e^{\al_2} [\cO_{\sB}] Q^{\al_1^{\vee}} + e^{\al_2} [\cO_{\sB (s_2 s_1)}] - e^{\al_2} [\cO_{\sB (s_2)}] Q^{\al_1^{\vee}}.$$
Applying Theorem~\ref{qKquot}, we deduce
\begin{align*}
\hskip -5mm [\cO_{\sB _{\{1\}}}] \star [\cO_{\sB _{\{1\}}}] & = [\cO_{\sB_{\{1\}}}],\\
\hskip -5mm [\cO_{\sB _{\{2\}} ( s_1 )}] \star [\cO_{\sB _{\{2\}} ( s_1 )}] & = ( 1 - e^{\al_2} ) [\cO_{\sB _{\{2\}} ( s_1 )}] + e^{\al_2} [\cO_{\sB _{\{2\}} (s_2 s_1)}].
\end{align*}
\item We have $[\cO_{\sB ( s_1 )}] \star [\cO_{\sB ( s_2 )}] = [\cO_{\sB ( s_1 s_2 )}] + [\cO_{\sB ( s_2 s_1 )}] - [\cO_{\sB (w_0)}]$. From this, we deduce
\begin{align*}
\hskip -5mm [\cO_{\sB _{\{1\}}}] \star [\cO_{\sB _{\{1\}} (s_2)}] & = [\cO_{\sB_{\{1\}} (s_2)}],\\
\hskip -5mm [\cO_{\sB _{\{2\}} ( s_1 )}] \star [\cO_{\sB _{\{2\}}}] & =[\cO_{\sB _{\{2\}} ( s_1 )}].
\end{align*}
\item We have $[\cO_{\sB ( s_1 )}] \star [\cO_{\sB ( s_1 s_2 )}] = ( 1 - e^{\al_2} ) [\cO_{\sB ( s_1 s_2 )}] + e^{\al_2} [\cO_{\sB (w_0)}]$. From this, we deduce
\begin{align*}
\hskip -5mm [\cO_{\sB _{\{1\}}}] \star [\cO_{\sB _{\{1\}} ( s_1 s_2 )}] & = [\cO_{\sB _{\{1\}} ( s_1 s_2 )}],\\
\hskip -5mm [\cO_{\sB _{\{2\}} ( s_1 )}] \star [\cO_{\sB _{\{2\}} ( s_1 )}] & = ( 1 - e^{\al_2} ) [\cO_{\sB _{\{2\}} ( s_1 )}] + e^{\al_2} [\cO_{\sB _{\{2\}} (s_2 s_1)}].
\end{align*}
\item We have $[\cO_{\sB ( s_1 )}] \star [\cO_{\sB ( s_2 s_1 )}] = ( 1 - e^{\vartheta} ) [\cO_{\sB ( s_2 s_1 )}] + e^{\vartheta} [\cO_{\sB (s_2)}] Q^{\al_1^{\vee}}$. From this, we deduce
\begin{align*}
\hskip -5mm [\cO_{\sB _{\{1\}}}] \star [\cO_{\sB _{\{1\}} ( s_2 )}] & = [\cO_{\sB _{\{1\}} ( s_2 )}],\\
\hskip -5mm [\cO_{\sB _{\{2\}} ( s_1 )}] \star [\cO_{\sB _{\{2\}} ( s_2 s_1 )}] & = ( 1 - e^{\vartheta} ) [\cO_{\sB _{\{2\}} ( s_2 s_1 )}] + e^{\vartheta} [\cO_{\sB _{\{2\}}}] Q^{\al_1^{\vee}}.
\end{align*}
\item We have
$$\hskip -10mm [\cO_{\sB ( s_1 )}] \star [\cO_{\sB ( w_0 )}] = ( 1 - e^{\vartheta} ) [\cO_{\sB ( w_0 )}] + e^{\vartheta} ( [\cO_{\sB}] Q^{\vartheta^{\vee}} + [\cO_{\sB (s_1s_2)}] Q^{\al_1^{\vee}} - [\cO_{\sB(s_1)}] Q^{\vartheta^{\vee}} ).$$
From this, we deduce
\begin{align*}
\hskip -5mm [\cO_{\sB _{\{1\}}}] \star [\cO_{\sB _{\{1\}} ( s_1 s_2 )}] & = [\cO_{\sB _{\{1\}} ( s_1 s_2 )}],\\
\hskip -5mm [\cO_{\sB _{\{2\}} ( s_1 )}] \star [\cO_{\sB _{\{2\}} ( s_2 s_1 )}] & = ( 1 - e^{\vartheta} ) [\cO_{\sB _{\{2\}} ( s_2 s_1 )}] + e^{\vartheta} [\cO_{\sB _{\{2\}}}] Q^{\al_1^{\vee}}.
\end{align*}
\end{itemize}
In all cases, the above calculations recover~\cite[Corollary~10]{BCLM19} as:
$$1 \star 1 = 1 \in qK_H ( \sB_{\{1,2\}} ) \equiv qK_H ( G / G ) = K_H ( G / G ) = \C P$$
by setting $[\cO_{\sB ( w )}] \equiv 1 \equiv Q^{\al_i^{\vee}}$ ($w \in W, i = 1,2$).

\medskip

{\small
\hskip -5.25mm {\bf Acknowledgement:} The author would like to thank Mark Shimozono for helpful conversations and Leonard Mihalcea for helpful discussions. The author also expresses gratitude to Thomas Lam and Changzheng Li for their comments. This research is supported in part by JSPS KAKENHI Grant Numbers JP26287004 and JP19H01782, and by a JSPS Grant-in-Aid for Challenging Research (Exploratory) 24K21192.}

{\footnotesize
\bibliography{ref}
\bibliographystyle{hplain}}
\end{document}